\documentclass{article}
\usepackage[utf8]{inputenc}

\usepackage{amsmath,dsfont,amsthm,url,datetime,subcaption,cases,mathtools,amssymb,natbib,bbm,bm}

\usepackage[pdfencoding = auto,psdextra]{hyperref}
\usepackage[margin=1.2in]{geometry} 
\usepackage{tikz}
\usepackage{subfiles}
\usepackage{enumitem}

\usepackage{pgfplots}
\usepackage{pgf-pie}

\definecolor{tikzcolor1}{HTML}{238716}
\definecolor{tikzcolor2}{HTML}{BA2FAA}
\definecolor{tikzcolor3}{HTML}{F59F4E}

\usepackage{todonotes}
\usepackage{macros}
\usepackage{cleveref}
\usepackage[linesnumbered,lined,ruled,commentsnumbered]{algorithm2e}

\newtheorem{lemma}{Lemma}
\newtheorem{theorem}{Theorem}
\newtheorem{cor}{Corollary}

\newtheorem{example}{Example}

\newtheorem{definition}{Definition}

\usepackage{tikzsymbols}

\title{Optimal Posterior E-values with Non-Convex Parameter Sets with Applications to Voting Systems}
\date{}
\author{Adrienne Tuynman and Timothée Mathieu}
\begin{document}
		\maketitle
                \begin{abstract}
                  We are interested in conducting political polls sequentially, so that one can stop acquiring data as soon as possible while safely yielding statistically significant results. Building off e-values, which have recently become a useful tool to create sequential testing methods, we develop a theory of posterior optimal e-values. We use voting as a convenient example on which to illustrate our method.

                  First, we design statistical tests for Condorcet and Borda voting system, and also for Schulze voting system which we are the first to tackle statistically. Then, we study the construction of optimal sequential e-values in the deceptively simple setting of multivariate Bernoulli data, with general composite null and alternative hypothesis sets $\mathcal{H}_0$ and $\mathcal{H}_1$. We give a way to compute these e-values using an efficient Frank-Wolfe algorithm, giving a pretty general way to compute Reverse Information Projections, even when $\mathcal{H}_0$ corresponds to a non-convex parameter set. Finally, we illustrate the efficiency, both in terms of power and sample size of our method. We compare with state of the art in both simulated and real data experiments, with application to French 2022 presidential election data. 

		\end{abstract}

	\section{Introduction}

We can only test what we can measure. In some fields, that is not an issue. It is easy to measure how far a robot can run, how many patients recover after being given a certain drug, or how much corn grows in a field after using a specific seed. The solution is to craft new numbers. It might not be feasible to ask patients to rate how good they feel about a treatment (such information may be very noisy, and not objective at all), it is way easier to ask them to compare two alternatives. During a poll, it is hard to ask voters to give politicians a grade out of ten as different citizens will not be using the same grading curve, but it is easy to ask them to compare two candidates. The frequency of one option being preferred to another is something that can be computed, estimated, and that is useful. We focus here on a simple problem. With a finite set of $N$ candidates, we receive comparison data: for two options $i$ and $j$, we are told how many times $i$ was preferred to $j$ and vice versa.

Our goal in this article will be to design sequential hypothesis tests based on e-values in the context of preference data. The idea is to collect data $X^1,X^2,\dots$ sequentially, and to stop as soon as we can take a decision. The decision is either to say that candidate $i$ is a winner for the chosen voting system or that he is not a winner. 

Our goals are twofold. First we design statistical tests for Condorcet and Borda voting system, and also for Schulze voting system which we are the first to tackle statistically. Our second goal is to study the construction of optimal sequential e-values in the deceptively simple setting of multivariate Bernoulli data with composite (and possibly non-convex) null and alternative hypothesis sets $\Hyp_0$ and $\Hyp_1$. The setting of votes is particularly adapted to display a variety of form for $\Hyp_0$ and $\Hyp_1$, with Condorcet being the easiest to the more complicated Schulze voting method. Contrary to adaptive settings such as duelling bandits (see \cite{bengsPreferencebasedOnlineLearning2021} for a survey on that topic), we do not choose which two options are compared. We will assume that we receive the same number of comparisons for every pair of options. 

We first show how our preference voting motivation can be cast as a statistical hypothesis problem on the preference matrix. Then, we develop a new (to our knowledge) methodology to construct e-variables as an extension of the GRO criterion to composite $\Hyp_1$. Using a Bayesian point of view, we assign a prior over $\Param_1$ the set of parameters of distributions in $\Hyp_1$ and define
$$G_t(E,\prior,\tau):=\E_{\param \sim \prior}\left[ \E_{X^{:t} \sim P_{\param}^{\otimes t}}\left[ \log(E(X^{:t\wedge\tau} )) \ind\{\param \in \Param_1\} \right]\right],$$
where $\prior$ is a prior on $\Param_1$, $\tau$ is an arbitrary stopping time and $X^{:t}:=X^1,\dots,X^t$ is an i.i.d. sample from $P_\param$. We show that among test supermartingales, there exists one which optimises $G_t(E,\prior,\tau)$. We then give its expression as product of one-step e-values, each a one-step GRO e-value with respect to the posterior distribution at time $t$. We call the resulting e-value POE (Posterior Optimal E-value). The POE is defined quite generally and is not specific to our voting scenario. 

In a second time, we show how to use POE to design both a fixed confidence and fixed sample size sequential test, and we give theoretical guarantees for both of these tests showing optimality in terms of expected stopping time. To this end, we also give concentration inequalities showing that POE converges to the GRO e-variable asymptotically. Then, we give a version of Frank-Wolfe algorithm adapted to our problem that allow us to compute the KL projection needed to define POE. This Frank-Wolfe algorithm is quite general and could have an independent interest as a way to perform a Reverse Information Projection when the distributions have finite support.

Finally, we illustrate the usage of POE for our setting in both simulated and real data experiments. In a simulated setting we compare the empirical power and stopping time of our tests to the state of the art (Universal Inference~\citep{wassermanUniversalInference2020}, GLRT and the e-value from \cite{turnerExactAnytimevalidConfidence2022}), we show how our method can in addition benefit from the knowledge of prior data. To conclude, we apply our method to political poll, using data from 2022 French Presidential elections, in an attempt to answer to the question ``who would have been elected president if French citizens had used Borda voting method?''.

        \paragraph{Prior works in statistics of social choice}

Several studies have investigated on statistical hypothesis testing on voting system, notably \cite{taplinStatisticalAnalysisPreference1997} use generalized likelihood ratio test to test for Condorcet winner. More recently \cite{xiaOptimalStatisticalHypothesis2020} investigate Neyman-Pearson type results Condorcet voting as well as for Mallows model.

A more parametric approach using ranking models, such as Bradley-Terry model, have also been used to study voting~\cite{makurMinimaxHypothesisTesting2024,sahaPACInstanceOptimalSample2020,rastogiTwoSampleTestingRanked2020}. Most of these works are centered on testing for Condorcet winner, and a few on other types of voting.
\paragraph{Related works in duelling bandits and online preference learning}

In duelling bandits~\cite{chenCombinatorialPureExploration2020,haddenhorstTestificationCondorcetWinners2021,sahaVersatileDuelingBandits2022,haddenhorstIdentificationGeneralizedCondorcet2021}, the problem is more active and the statistician decides on which pair of preference to get next. See in particular the survey in \cite{bengsPreferencebasedOnlineLearning2021} to get an overview on duelling bandits. However, obtaining optimality results as strong as the ones we present in this article remains hard in general in the duelling bandit setting.

\paragraph{Prior works in sequential testing with e-values}

General introduction to e-values can be found in~\cite{ramdasHypothesisTestingEvalues2025} as well as in the survey \cite{ramdasGameTheoreticStatisticsSafe2023}. Most work on e-values for voting are concentrated around auditing~ \cite{spertusSweeterSUITESupermartingale2022,waudby-smithRiLACSRiskLimitingAudits2021}, which does not align with the purpose of this article. 

The work closest to our own is the one of Turner and Grünwald in~\cite{turnerExactAnytimevalidConfidence2022} (see also~\cite{turnerGenericEvariablesExact2024}) in which, building on the theory of Growth Rate Optimal (GRO) e-values from \cite{grunwaldSafeTesting2023}, the authors study a construction of approximate GRO e-values that can be applied in our setting (albeit in dimension 2). Our method (POE) can be seen as an extension of \cite{turnerExactAnytimevalidConfidence2022,turnerGenericEvariablesExact2024} to more general non-convex and multivariate parameter sets. Note that compared to \cite{turnerExactAnytimevalidConfidence2022,turnerGenericEvariablesExact2024} we asks additionally that each sample give the complete preferences (i.e. the sample size for each Bernoulli is the same). 

More recently, in the case of simple alternative hypothesis, \cite{grunwaldOptimalEValuesExponential2025,haoEValuesExponentialFamilies2025} investigate the case of exponential families, giving condition under which the solution of the Reverse Information Projection (RIPr) problem is simple and result in a likelihood ratio test between with a most-confusing instance belonging to $H_0$. See also recent works on testing against a composite alternative in \cite{arnoldOptimalEvaluesTesting2026} in the context of testing by betting on the mean of distributions.

        \section{Statistical test for preference voting}
Our first goal is to set-up a statistical testing problem of the form $\Hyp_0 = \{P_\param^{\otimes T}, \param \in \Param_0\}$ against $\Hyp_0 = \{P_\param^{\otimes T}, \param \in \Param_1\}$ and decide whether data are sampled from a distribution in $\Hyp_0$ or in $\Hyp_1$. We first give the formal definition of the different $\Hyp_0$ we will consider in this article, drawn from preference voting systems. 
\subsection{Different kinds of winners}

\begin{definition}[Preference matrix]\label{def:matrixpref}
	Sampling a comparison between options $i$ and $j$ means sampling from a Bernoulli of parameter $\pr{i,j}$, $\Ber(\pr{i,j})$, where $\pr{i,j}$ is the probability that $i$ is preferred to $j$. An instance with $N$ options is thus defined by a preference matrix $\PrM = \left(\pr{i,j}\right)_{1\leq i,j\leq N}$, where for any $i\neq j$, $\pr{i,j}=1-\pr{j,i}$. $\pr{i,i}$ is defined, arbitrarily, as $0.5$.
\end{definition}

While the matrix is of size $N\times N$, it is uniquely defined by the $\pr{i,j}$ with $i<j$. Therefore, complete samples will be of size $D={N\choose 2}$, and $D$ will be referred to as the dimension of the problem.

In the rest of the article, we refer to $\Param^i$ as the set of preference matrices such that $i$ is a winner for the chosen voting method, and to $\Param$ as the set of all possible preference matrices. In \Cref{fig:constraints} we plot  $\Param^1$ the parameter sets in $\Param$ of the three voting systems that we describe below when $D=N=3$.

\begin{figure}[h]
  \begin{center}
     \subfloat[Condorcet\label{fig:condorcet}]{\includegraphics[width=0.33\textwidth,trim={10cm 2cm 10cm 2cm},clip]{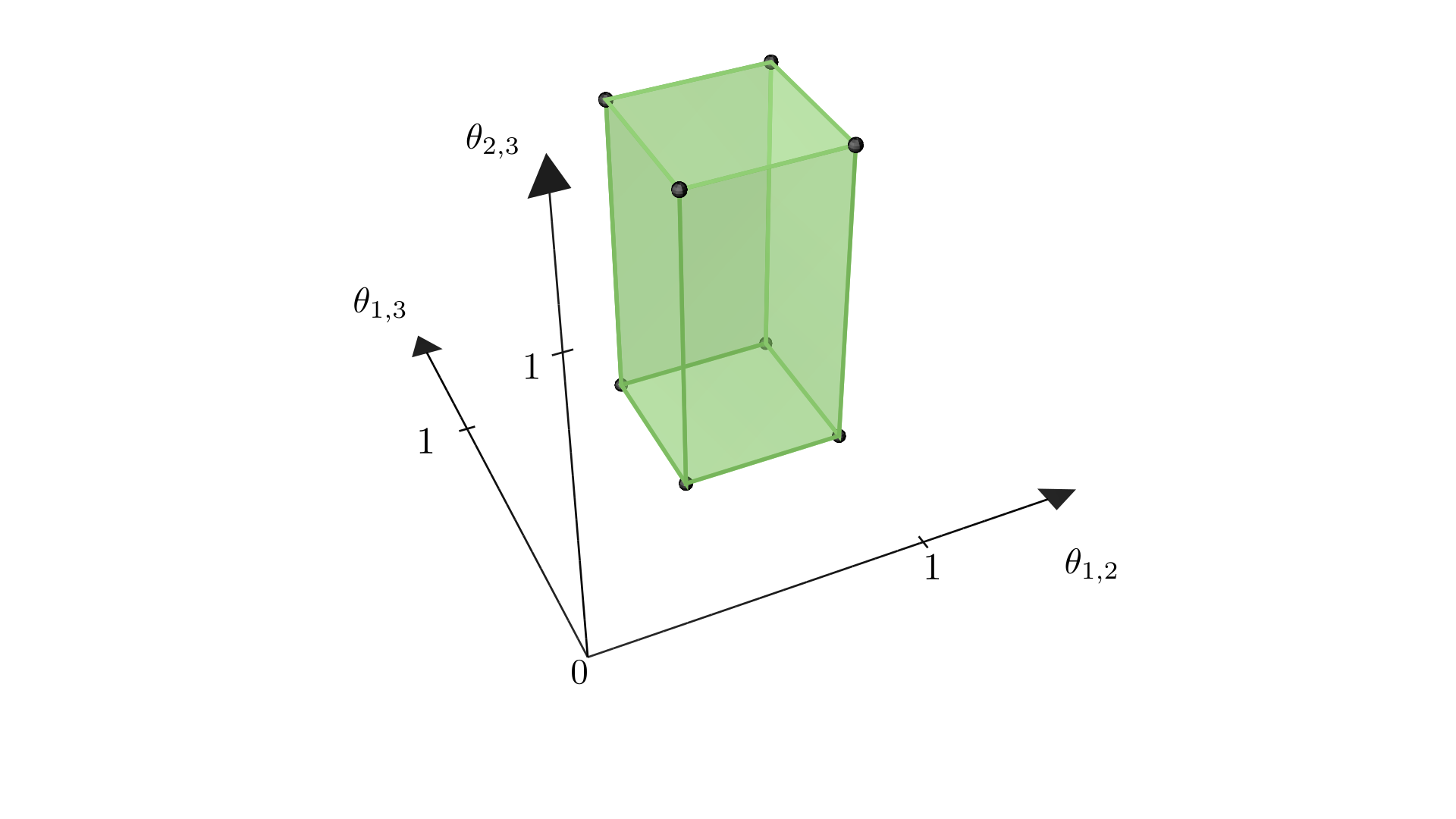}}
    \subfloat[Borda\label{fig:borda}]{\includegraphics[width=0.33\textwidth,trim={10cm 2cm 10cm 2cm},clip]{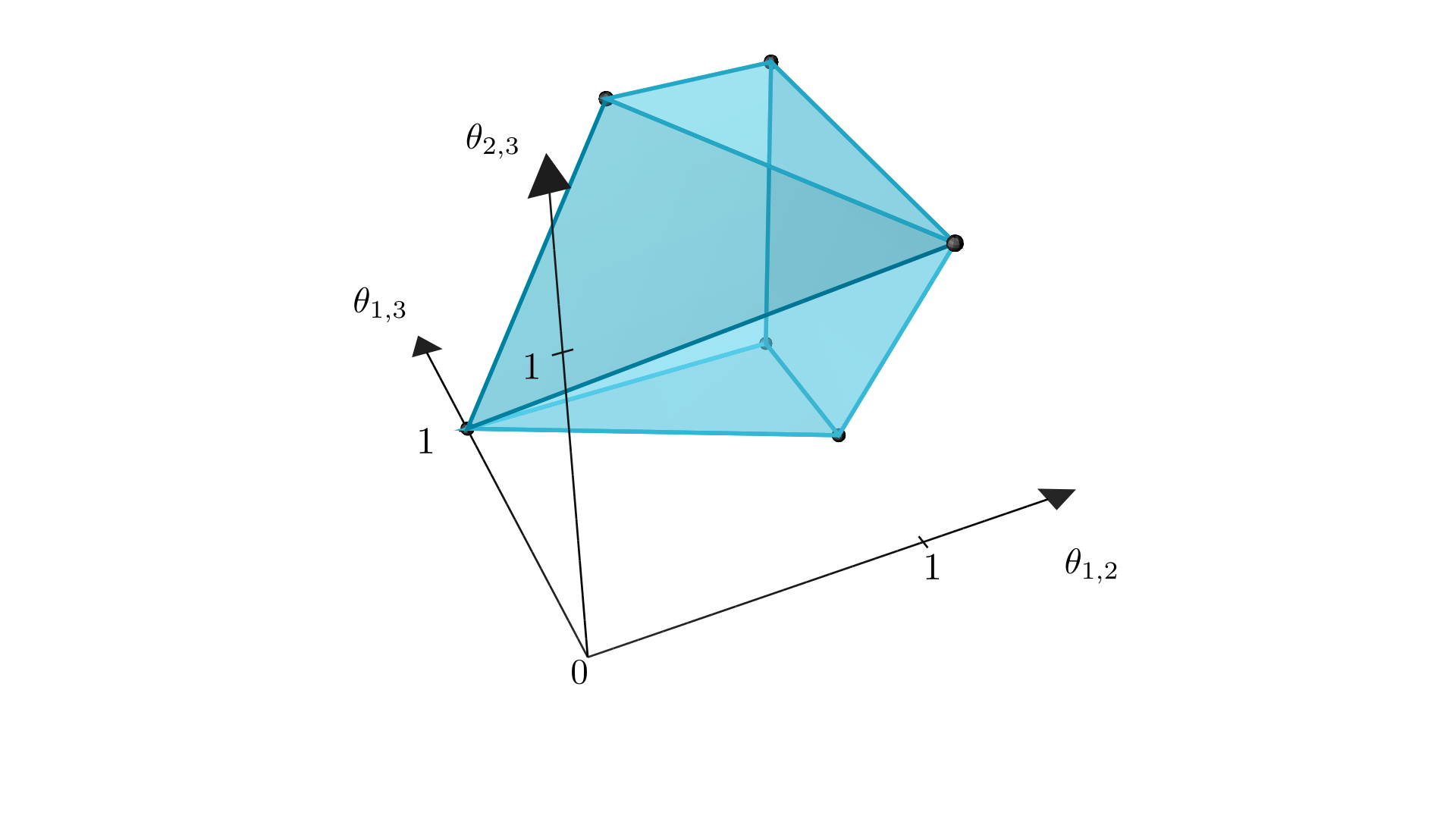}}
    \subfloat[Schulze\label{fig:schulze}]{\includegraphics[width=0.33\textwidth,trim={10cm 2cm 10cm 2cm},clip]{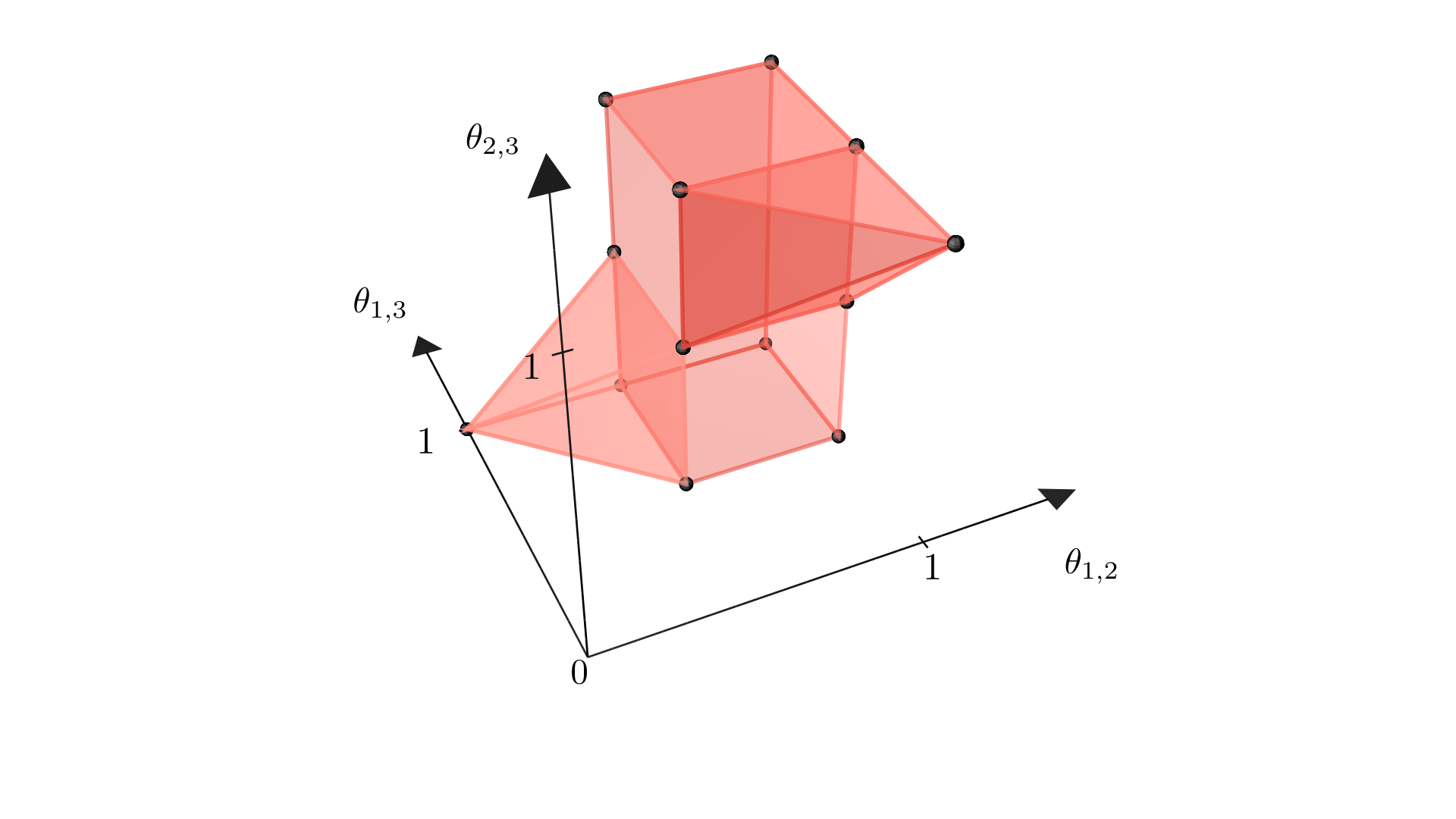}}
  \end{center}
\caption{Representation of the parameter set $\Param^1$ for which $1$ is winner, for the three voting systems we consider. \label{fig:constraints}}
\end{figure}

\subsubsection{Condorcet}

When discussing preference voting, the first option that comes to mind is the Condorcet winner. The concept is simple: if an option $i$ is preferred to any option $j$, then it is a Condorcet winner.

\begin{definition}[Condorcet]
	If for an option $i^\star$, \[\forall i\neq i^\star,\, \pr{i^\star,i}\ge 1/2 \] then $i^\star$ is a Condorcet winner.
\end{definition}

In \Cref{fig:condorcet}, we show what $\Param^1$ looks like with three candidates. This definition makes sense, but has a major drawback: the Condorcet winner does not always exist, i.e. $\bigcup_{i=1}^N \Param^i \neq \Param$. We can spot that in the figure: adding up thrice $\Param^1$ rotated to represent $\Param^2$ and $\Param^3$ will not yield the entire $[0,1]^3$ cube. This is what we call Condorcet's voting paradox.

We add one more condition to consider: we want the result to be determined uniquely by the preference matrix $\PrM$. This allows us to only use this matrix to quantify the election results, and to conduct tests based on subsets $\Param^i$ of matrices, see~\Cref{sec:instant_runoff} for an example of voting system that cannot be deduced from the preference matrix.

\subsubsection{Borda}

In the Borda system, voters assign points to the candidates: $N-1$ points to the candidate they prefer, then $N-2$ to the second, all the way until $0$ points for the candidate they least like. The winner is the candidate who received the most total points. 
While this formulation may sound like the precise order is important, another way to look at it is that each won duel grants one point.

\begin{definition}
	If for an option $i^\star$, \[\forall i\neq i^\star, \sum_{j}\pr{i^\star,j} \geq \sum_j \pr{i,j}\] then $i^\star$ is a Borda winner.
\end{definition}

We represent $\Param^1$ in \Cref{fig:borda}. While there might be tied candidates, the set of preferences matrices $\PrM$ inducing ties is negligible (when selecting a preference matrix uniformly at random, drawing one with multiple winners almost surely does not happen); and Borda winners can be determined from the preference matrix. This makes it a great candidate for efficient voting. Moreover, the Borda procedure also has an important property: the winner set is convex.

\begin{lemma}
	For $i\in [N]$, the set $\Param^i$ of matrices $\PrM$ that describe a preference vote where $i$ is a Borda winner is convex.
\end{lemma}
The proof follows directly from the definition of Borda winner has a linear condition which implies that Borda winner is an intersection of half-spaces, hence it is convex. 

However, if the Condorcet winner exists, it might not necessarily be elected by a Borda process, as we can see in the figure.
We will therefore explore one last voting method.

\subsubsection{Schulze}

In the Condorcet winner method, it does not matter by \emph{how much} a candidate $i$ beats another candidate $j$; all that matters is that it does beat it. Moreover, if $j$ goes on to beat a candidate $\ell$, then the Condorcet method does not deduce anything about the relationship from $i$ to $\ell$. The Schulze method aims to fill these two gaps.

\begin{definition}[The Schulze method \citep{schulzeNewMonotonicCloneindependent2011}]
	A path $c$ from candidate $i$ to candidate $j$ is a sequence of candidates $c_1,\dots,c_n$ such that $i=c_1$, $j=c_n$, and for any $\ell,m\in [n]$, if $\ell\neq m$, then $c_\ell\neq c_m$. We denote $i\rightarrow j$ the set of paths from $i$ to $j$. The strength of path $c$ is \[\mathrm{Str}(c) = \min_{m\in [n-1]} \pr{c_m,c_{m+1}} . \] In other words, a path is only as strong as its weakest link.
	We say that $i$ beats $j$ in the Schulze method if \[\max_{c\in i\rightarrow j} \mathrm{Str}(c) \geq \max_{c\in j\rightarrow i} \mathrm{Str}(c) .\]
	$i$ is the Schulze winner if it beats every other candidate. In particular, a Condorcet winner is necessarily a Schulze winner.
\end{definition}

While the Schulze winner is determined by the preference matrix, it is difficult to represent $\Param^i$ with this definition. In fact, it is not even convex, as we will see numerically. However, we find that it is possible to define convex subsets of $\Param$ that share a Schulze winner, through defining orderings. Indeed, what matters for the Schulze method is not how much a path beats another, but merely if it does; meaning that what matters is the ordering of $\pr{i,j}$. We summarise this observation, and the consequences from it, in the following lemma.

\begin{lemma}
	For $\PrM$ and $\PrM'$ two preference matrices, we say that they share an ordering if
	\begin{itemize}[itemsep=0.1pt,parsep=0.1pt,topsep=0.1pt]
		\item for all $i,j,k,l$ such that $i\neq j$ and $k\neq l$, we have $\pr{i,j}\geq \pr{k,l} \iff \pr{i,j}'\geq \pr{k,l}'$
		\item for all $i\neq j$, $0.5 \geq \pr{i,j} \iff 0.5 \geq \pr{i,j}'$.
	\end{itemize}
	If they do, then $\PrM$ and $\PrM'$ share a Schulze winner. 
	
	Moreover, if $\PrM$ is such that, for any $i\neq j$ and $k\neq l$, $(i,j)\neq (k,l) \implies \pr{i,j}\neq\pr{kl}$ and $\pr{i,j}\neq 0.5$ (it has a strict ordering), then $\PrM$ admits only a single Schulze winner.
	
	Finally, for any such matrix, the set of preference matrices that share an ordering with it is convex.
\end{lemma}

This gives us a way to describe $\Param^i$ as a reunion of convex spaces, one per possible strict ordering. Of course, not all orderings are possible, due to $\pr{i,j}=1-\pr{j,i}$. However, It is possible to count the number of possible orderings, by noticing that each ordering can be described exactly by the order of $\pr{\{i,j\}} = \min\{\pr{i,j},\pr{j,i}\}$, and for each $i,j$, which one is bigger than the other. The number of possible strict orderings is therefore $D! \times 2^D$. As $D$ grows, the amount of convex components grows very quickly, which is a challenge for any statistical testing method.

In \Cref{fig:schulze}, we see $\Param^1$ for $N=3$ candidates. Note that, although we then have $48$ possible strict orderings, thus $16$ orderings in $\Param^1$, we only see three convex components in this figure. Indeed, if $\pr{1,2}>0.5$ and $\pr{1,3}>0.5$, then we already know that 1 is Condorcet winner, and thus Schulze winner; the other comparisons do not matter, and thus we can merge strict orderings together. This approach of merging orderings together should be applicable to higher dimensions; however, we have no optimal way of conducting the mergings, and heuristic approaches have so far failed to significantly reduce the number of convex components. As such, in this article, we will limit ourselves to studying Schulze method for $N=3$.

\section{E-values for preference voting}

\subsection{Basics of e-values}

A recent way to conduct sequential hypothesis testing are e-variables, which are defined in a deceptively simple way.

\begin{definition}[E-variable]
	Let $Y\in\mathcal{Y}$ be the data sampled from an unknown distribution $P$. A nonnegative random variable $E:\mathcal{Y}\rightarrow [0,\infty]$ is an e-value for the null hypothesis $\Hyp_0$ if, for any distribution $P\in \Hyp_0$, \[ \E_{Y\sim P}[E(Y)]\leq 1 .\]
\end{definition}

E-values are more resilient than $p$-values. First of all, they are easier to interpret: since under $\Hyp_0$, we expect to see a value smaller than 1, then a big e-value makes it more likely for us to be under $\Hyp_1$. The e-value can therefore be seen as the amount of evidence against $\Hyp_0$. That makes them easy to explain and interpret, even without a threshold. 
Secondly, it becomes easy to natively combine independent experiments. For that, it suffices to notice that the product of two e-values remains an e-value. See \cite{grunwaldNeymanPearsonEvaluesEnable2024} or \cite[Section 1.2]{ramdasHypothesisTestingEvalues2025} for a more in-depth comparison between e-values and p-values.

To have useful e-values, that will have some power, it is good to have ``big'' e-values under $\Hyp_1$; e-values that are bigger than all the other ones under $\Hyp_1$. One way to formalise this is to look for e-values that are \emph{numeraire e-values}: for $Q$ a distribution, we consider $E^\star$, which satisfies for any other e-value $E$ that $\E_Q[E/E^\star]\leq 1$. This is equivalent to maximising the expectation under $Q$ of $\log E$ \cite[Proposition 6.3]{ramdasHypothesisTestingEvalues2025}, called the e-power \cite[Definition 3.11]{ramdasHypothesisTestingEvalues2025}. This log-optimality criterion might seem like an arbitrary choice; however, since we often want to multiply e-values together to yield other e-values, log-optimality is more convenient to consider than, for example, maximising the expectation of $E$ itself. 

The following result guarantees the existence of the optimal e-value under singleton $\Hyp_1$. 

\begin{theorem}[GRO, from~\cite{grunwaldSafeTesting2023}]\label{th:gro}
	Assume that $\Hyp_1=\{P_{\param_1}\}$ is a singleton. Then there exists a unique distribution $P_0^\star$
	such that \[E^\star := \frac{\d P_{\param_1}(Y)}{\d P_0^\star(Y)}\] is an e-variable. 
	Additionally, $E^\star$ satisfies, essentially uniquely, \[\sup_{E} \E_{Y\sim P_{\param_1}} [\log E] = \E_{Y\sim P_{\param_1}}[\log E^\star] = \inf_{W_0\in \Prior_0} \KL(P_{\param_1},P_{W_0}) = \KL(P_{\param_1},P_0^\star)\] where the sup over $E$ is taken in the set of e-values over $\Param_0$. 
	
	$E^\star$ is called the Growth Rate Optimal (GRO) e-value for $\param_1$, or the numeraire e-value.
\end{theorem}

\Cref{th:gro} holds under the hypothesis that $P_{\param_1}$ has full support and that $\inf_{W_0\in \Prior_0} \KL(P_{\param_1},P_{W_0})$ is finite. Furthermore, $P_0^\star$ can potentially be a subdistribution. Since those hypotheses are always satisfied in our setting, and since $P_0^\star$ is always a distribution for us, we will omit those details in the following.

The minimisation in KL found in \Cref{th:gro} can be interpreted as a projection of the alternative distribution onto $\Hyp_0$, this projection is called the Reverse Information Projection (RIPr) and has an important role in information theory~\citep{cover1999elements}.

\subsection{Turner-Grünwald approach to approximate GRO}

In our setting $\Param_1$, the parameters for $\Hyp_1$, is not a singleton, so that the GRO construction cannot be carried away right away and the $\E_Q[\log E]$ criterion is not adapted. A first way to solve this problem was found in \cite[Theorem 2]{turnerExactAnytimevalidConfidence2022} in which the authors estimate the alternative using bayesian posterior distribution and project this posterior on $\Hyp_0$ using simplified form for the KL in this setting. Their approach depend on the fact that the GRO e-value $P_0^\star$ can be taken as $P_\param$, instead of general $P_{W_0}$ in the specific case of Bernoulli distribution when $\Param_0$ is convex and closed, so that the only thing that needs to be found is the value of $\param$. This type of results can be found in \cite[Theorem 2]{turnerExactAnytimevalidConfidence2022} for 2-dimensional Bernoulli distributions, or in \cite{grunwaldOptimalEValuesExponential2025} for exponential families. We derive our own result for our setting, multivariate Bernoulli distributions of dimension $D$.

\begin{lemma}\label{lem:pointpoint}
	Let $\psi \in (0,1)^D$ and $\Param'\subset [0,1]^D$ be a convex closed set. We choose to denote $\mathcal{P}_{\Param'}^T = \left\{ P_{\phi}^{\otimes T},\, \phi\in \Param'\right\}$. 
	Then, there exists $\phi\in \Param'$, independent of $T$, such that, with $Y^{:T} = Y^1,\dots,Y^T$, $$E^\star \left( Y^{:T}\right)= \frac{\d P_{\psi}^{\otimes T}(Y^{:T})}{\d P_{\phi}^{\otimes T}(Y^{:T})} = \prod_{t=1}^T  \frac{\d P_{\psi}(Y^{t})}{\d P_{\phi}(Y^{t})};$$ in other terms, the RIPr of $P_\psi^{\otimes T}$ onto $\Hyp_0$ is $P_W^{\otimes T}$ with $W$ a Dirac distribution. This also means that the GRO e-value over $T$ steps is the product of the GRO e-values over one step.
\end{lemma}

\Cref{lem:pointpoint} comes from the characterisation of the GRO e-value as $\frac{\d P_\psi^{\otimes T}}{\d P}$ with $P$ the RIPr of $\d P_\psi^{\otimes T}$ onto $\Hyp_0$, that is, the distribution minimising $\KL(P_\psi^{\otimes T},P)$ and from \Cref{th:gro}. The proof for \Cref{lem:pointpoint} can be found in \Cref{sec:proof_pointpoint}.

Using this Lemma~\ref{lem:pointpoint} we can define what we call later the TG e-variable which consider at time $t$, the (fictive) alternative $\Hyp_1^t = \{P(\cdot \mid X^{:t})\}$ comprised only of the posterior distribution and use it to compute the GRO e-value for each individual step using Lemma~\ref{lem:pointpoint}. Remark that this approach depend on the convexity of $\Param_0$ and cannot be used to test for Schulze winner. On the other hand, using a Beta prior on $[0,1]^D$, \cite{turnerExactAnytimevalidConfidence2022} also simplify the computation of the KL to end up with a computation of a minimum of sum of KL on each dimension. See \Cref{sec:experimentation} for a practical comparison of TG e-values with our work.

	\subsection{The POE e-value}

Inspired from Bayesian decision theory \cite[Section 2.3]{robertBayesianChoiceDecisionTheoretic2007} and by the definition of the GRO e-value, we define the utility function of an e-value $E$ as 
$$U(E, \param) = \log(E) \ind\{\param \in \Param_1\}.$$
Remark that we place ourselves in the context of utility function instead of losses to fit the setting of \cite{grunwaldSafeTesting2023}. We aim to find the e-value $E_t$ which, for any stopping time $\tau$, optimises its stopped integrated utility $G_t$ with respect to a prior $\prior$. We define
\begin{align}\label{eq:gn}
	G_t(E,\prior,\tau) &:=  \E_{\param \sim \prior}\left[ \E_{X^{:t} \sim P_{\param}^{\otimes t}}\left[ \log(E(X^{:t\wedge\tau} )) \ind\{\param \in \Param_1\} \right]\right].
\end{align} 

Remark that for $\prior$ a prior with support on $\Param_1$ this simplifies as the indicator function is not needed anymore. In what follows we suppose that the support of $\prior$ is in $\Param_1$.

We will search for an optimal e-variable that is also a test supermartingale, i.e. e-variables such that for any $t$ and $\param \in \Param_0$, we have $\E_{X^{:t} \sim P_{\param}^{\otimes t}}[E(X^{:t})\mid X^{:t-1}]\le E(X^{:t-1})$. While not all anytime e-values can be built using them~\citep{koolenLogoptimalAnytimevalidEvalues2022}, test supermartingales are an easy and well-behaved way to build them.

Among those e-values, we search for an e-value which maximises for each $t$ the quantity $G_t(E,\prior,\tau)$, which is the growth rate for some prior $\prior\in \Prior_1$ and stopping time $\tau$. We call it the Posterior Optimal E-value (POE for short).

\begin{definition}[POE]\label{def:POE}
	Let $\prior\in \Prior_1$ be a prior over $\Param_1$. 
	$(E_t)$ is a Posterior Optimal E-value (POE) for $\prior$ if it is a test supermartingale and if, for any $(E_n')$ test supermartingale and for any stopping time $\tau$, \[ \forall t,\, G_t(E_t,\prior,\tau) \geq G_t(E_t',\prior,\tau) .\]
\end{definition}

While the condition that POE must optimise $G_t(E,\prior,\tau)$ for any stopping time $\tau$ might seem very constraining, we will see that there indeed exists an e-value which satisfies it.

Note that this is a generalisation of the GRO e-value of \Cref{th:gro}, as the definitions match when $\Param_1$ is a singleton and for $\tau=1$. For a constant stopping time, compared to the GROW e-value, we want to be good not for the worst case, but over multiple possible parameters - or, more accurately, over a distribution of parameters.

\begin{theorem}[POE is the product of one-step optimal posterior e-values]\label{th:bayesian}
	Let $W_{1,t}$ be the posterior distribution over $\Param_1$ at time $t$ defined by
        $$W_{1,t}(\param) := \frac{\prod_{i=1}^{t-1} P_{\param}(x_i) \prior(\param)}{\int \prod_{i=1}^{t-1} P_{\param}(x_i) \prior(\param) \d \param},$$ 
	and 
	$$W_{0,t} := \argmin_{W\in \Prior_0}\KL\left(\int P_{\param}(X^t) \d W_{1,t}(\param),\int P_{\param}(X^t) \d W(\param)\right).$$ 
	Assume that, for any $\lawa,\lawb\in \Param_0\cup\Param_1$, $\KL(P_\lawa,P_\lawb)<+\infty$.
	
	Then, we have that the POE e-variable exists and is equal to
	\begin{equation}\label{eq:POE} E_t(x^{:t}) = \prod_{i=1}^t \frac{\int P_{\param}(x_i) \d W_{1,i}(\param)}{\int P_{\param}(x_i) \d W_{0,i}(\param)}.\end{equation} 
\end{theorem}

The proof can be found in \Cref{sec:proof_bayesian}.
In other words $W_{0,t}$ minimises the one-sample posterior KL conditionally on the past. This approach is motivated by  \cite[Theorem 1]{grunwaldSafeTesting2023} applied on the posterior distribution at step $t$, which states that the chosen $W_{0,t}$ is the one-step optimal e-value at time $t$ conditionally on the past. 

Note that it is also similar to the e-value given by \cite[Equation (7)]{turnerExactAnytimevalidConfidence2022} for multivariate Bernoulli with $D=2$ dimensions, with a few differences. They define \Cref{eq:POE} merely as \emph{an} e-value; they focus on the multivariate Bernoulli $D=2$ case; and fix a particular prior which contrary to our case has support in the whole $\Param$ space. Meanwhile, our definition of POE is based off optimising the rate $G_t$; our definition of POE and \Cref{th:bayesian} is valid for any possible test in which the $\Hyp_i$ are defined as distributions parametrised by some sets $\Param_i$, as long as the $\KL$ over $\Param_0\cup\Param_1$ is finite; and our definition includes any kind of prior. 


	
\section{Statistical testing with POE}\label{sect:guaranteese}

In this section, we describe the two testing procedures that we consider in this article, and give their guarantees in power and sample complexity.

\paragraph{Capped sample size test}
The first procedure has a finite, pre-computed sample size that has the advantage of having a fixed maximum sample size. However, it has the disadvantage of being non-optimal on average in terms of stopping time, and it depends on some parameters of the model that may not be known in practice. 

More formally, to test $\Hyp_0$ against $\Hyp_1$ using samples $X^1,X^2,\dots$, we construct at each timestep $t$ an e-variable $E(X^{:t})$. We also decide on a maximum sample size $T$.

The Capped sample size test then consists in collecting data and stopping either for the first $t$ such that $E(X^{:t})\ge 1/\alpha$, in which case we reject $\Hyp_0$; or stopping when we collected $T$ data and never rejected before, at which point we accept $\Hyp_0$. We call this procedure stopped e-sequential test. In the sequel, we generally call this test according to which e-value it uses (e.g. POE test that uses POE e-value).

\paragraph{Fixed confidence test}
The second procedure is optimal in term of expected stopping time and does not depend on knowing anything about the model. Its drawback, however, is that the stopping time is not controlled almost surely and it may happen that the test takes a long time to stop.

To test $\Hyp_0$ against $\Hyp_1$ using samples $X^1,X^2,\dots$, we construct two sequences of e-variables $E(X^{:t})$ and $E'(X^{:t})$ such that 
$$\forall 1\le t, \forall P \in \Hyp_0,\quad  \E_P[E(X^{:t})]\le 1 ,$$
$$\forall 1\le t, \forall Q\in \Hyp_1,\quad  \E_Q[E'(X^{:t})]\le 1,$$
meaning that $E$ is a sequence of e-values for $\Hyp_0$ (against $\Hyp_1$), and $E'$ is a sequence of e-values for $\Hyp_1$ (against $\Hyp_0$).

The test is then to collect data and stop the first time one of two events occurs: either $E(X^{:t})\ge 1/\alpha$, in which case we reject $\Hyp_0$; or $E'(X^{:t}) \ge 1/\beta$, and we reject $\Hyp_1$. The stopping time is thus 
\[\tau = \inf\left\{ t\geq 1, \, E(X^{:t}) \geq \frac 1 \alpha \text{ or } E'(X^{:t})\geq \frac 1 \beta\right\} .\]
 Remark that in this testing procedure, we do not have an almost sure bound on the stopping time. When $E$ and $E'$ are constructed with POE, the resulting test is called the PPOE (parallel posterior optimal e-value) test.

\subsection{Some concentration results for POE and the likelihood ratio}
Let us start with some concentration results first.

\begin{lemma}[Concentration for likelihood ratio]\label{lem:GRO}
	Let $Q$ and $P$ be two multivariate Bernoullis of dimension $D$, of parameters respectively $\lawa$ and $\lawb$. Assume $X\sim Q$. Define $\tau = \min\{T \text{ such that }\prod_{t=1}^T \frac{\d Q}{\d P}(X^t) >\frac{1}{\alpha}\}$.
	
	We have for any $\nu>0$ and $\varepsilon\in (0,1)$ that, with probability greater than $1-\nu$, \begin{equation*}  \tau \le \frac{\log\frac{1}{\alpha} +m(\lawa,\lawb)}{(1-\varepsilon)\sum_{d\in [D]}  \kl(\lawa_d,\lawb_d)}+\frac{\log (1/\nu) \sum_{d\in [D]} \sigma_d^2}{8\varepsilon(1-\varepsilon)\left(\sum_{d\in [D]} \kl(\lawa_d,\lawb_d)\right)^2} \end{equation*} where $\sigma_d = \left| \log\frac{\lawa_d(1-\lawb_d)}{\lawb_d(1-\lawa_d)}\right|$ and $m(\lawa,\lawb)=\sum_{d\in [D]} \max\left\{  \log\frac{1-\lawa_d}{1-\lawb_d}, \log\frac{\lawa_d}{\lawb_d}\right\}$.
	
\end{lemma}
We show this technical result in \Cref{appsect:proofsGro}.

Note that, as will be proved in \Cref{appsect:mandsigma}, 
\begin{equation}\label{eq:mandsigma} 
m(\lawa,\lawb)\leq \sum_{d\in[D]} \sigma_d  
\end{equation} 
and 
\begin{equation}
\label{eq:sigmaandkl} \kl(\lawa_d,\lawb_d)\leq \sigma_d,
\end{equation}
In the more complicated case of the concentration of POE, we also have to handle the speed of convergence of the posterior distribution to the sampling distribution and this gives us the following theorem.

\begin{theorem}[Concentration for POE, general case]\label{th:POE}
	Assume that $\Param_0$ and $\Param_1$ are such that for any $\param\in \Param_0\cup \Param_1$, we have for all $d\in [D]$ that $\param_d \in (\varepsilon,1-\varepsilon)$. Assume $X\sim Q$ with $Q=P_\lawa$ a multivariate Bernoulli of dimension $D$ and let $\prior$ be a prior over $\Param_1$. Let $E_t$ be the e-variable defined in \Cref{eq:POE}.
	We have for any $\nu>0$ that, with probability greater than $1-2\nu$, 
	\begin{align*} \log \left(E_T(X^{:T})\right) &\ge T \KL\left(Q, P^\star \right)-1-\log\left(\frac{1}{\prior\left(\mathcal{B}_{\infty}\left( \left( 1-\frac 1{TD}\right)\lawa, \frac{1}{TD}\right)\right)}\right)  \\
		&\hspace{4em} -\log(1/\varepsilon)\sqrt{2 TD \log(1/\nu)}
	\end{align*}
	where $P^\star$ is the RIP of $Q$ on $\Hyp_0$. 
\end{theorem}

The proof of this lemma is to be found in \Cref{appsect:proofsPOE}.

This result depends on the prior chosen; therefore, we look at what happens in the case of a particular prior, the uniform one. 

\begin{cor}[Concentration for POE, uniform prior]\label{cor:POEuni}
	If, in addition to the assumptions of \Cref{th:POE}, $\prior$ is uniform on $\Param_1$, we get with probability greater than $1-2\nu$
	{\small
	$$\log \left(E_T(X^{:T})\right) \ge T \KL\left(Q, P^*\right)-1-D\log (TD)  -\log(1/\varepsilon)\sqrt{2 TD \log(1/\nu)}.
	$$
	}
\end{cor}

With those results, we are now ready to derive some guarantees for tests using GRO and POE.

\subsection{Capped sample size test}
For the Capped sample size test, we have the following guarantees.

\begin{theorem}[Fixed sample size test for GRO]\label{th:GRO}
	For $\Hyp_1 = \{Q\}$ with $Q=P_\lawa$, and $\Param_0$ convex,
	fix \[ T = 2\frac{\log\frac{1}{\alpha} +\sum_{d\in[D]} \log\frac{1-\lawa_d}{1-\lawb_d}}{\sum_{d\in [D]}  \kl(\lawa_d,\lawb_d)}+\frac{\log (1/\beta) \sum_{d\in [D]} \sigma_d^2}{2\left(\sum_{d\in [D]} \kl(\lawa_d,\lawb_d)\right)^2}\] as the maximum sample size, with $P=P_\lawb$ the RIPr of $Q$ onto $\Hyp_0$. Then the described test for GRO has power greater than $1-\beta$ and false positive rate smaller than $\alpha$.
	
	Moreover, the stopping time under $Q$ then satisfies \begin{equation} \label{eq:exptau} \E_Q[\tau] \le \frac{\log\frac{1}{\alpha} }{\sum_{d}  \kl(\lawa_d,\lawb_d)} + \cO\left(  \sqrt{ \frac{\log\frac{1}{\alpha} }{\sum_{d}  \kl(\lawa_d,\lawb_d)}} \right) \end{equation}
\end{theorem}

\begin{proof}
	From \Cref{lem:pointpoint}, with convex $\Param_0$, $P$ the RIPr of $Q$ onto $\Hyp_0$ can be written as $P=P_\lawb$ with $\lawb\in \Param_0$. Applying \Cref{lem:GRO} with $\varepsilon = \frac 1 2$ and applying \Cref{eq:mandsigma} yields the first result.
	
	The proof for \Cref{eq:exptau} can be found in \Cref{appsect:GROexp}.
\end{proof}

Note that \Cref{th:GRO} depends on the convexity of $\Param_0$. For non-convex $\Param_0$, we can still conduct the test for POE.

\begin{theorem}[Fixed sample size test for POE]\label{th:POE_samplesize}
	Assume that $\Param_0$ and $\Param_1$ are such that for any $\param\in \Param_0\cup \Param_1$, we have for all $d\in [D]$ that $\param_d \in (\varepsilon,1-\varepsilon)$. 
	Fixing $T$ as the sample size, where $T$ satisfies
	\begin{align*}
			T &\geq \frac{B +\log\frac{1}{\alpha}}{\KL(Q,P^\star)} + \frac{D\ln\frac{2}{\beta}\left(\log\frac{1}{\varepsilon}\right)^2}{\KL(Q,P^\star)^2} \\ 
			&\hspace{2em}+ \frac{\log\frac{1}{\varepsilon}\sqrt{D\ln\frac{2}{\beta}}}{\KL(Q,P^\star)^2}\sqrt{ 2\KL(Q,P^\star)\left(B +\log\frac{1}{\alpha}\right)  + D\ln\frac{2}{\beta}\left(\log\frac{1}{\varepsilon}\right)^2 } 
	\end{align*}
with $B =1+\log\left(\frac{1}{\prior\left(\mathcal{B}_{\infty}\left( \left( 1-\frac 1{TD}\right)\lawa, \frac{1}{TD}\right)\right)}\right)$,
	the described test for POE has power greater than $1-\beta$, and false positive rate smaller than $\alpha$.

\end{theorem}
The proof of this result can be found in \Cref{appsect:fixedPOE}. Note that \Cref{th:POE} can be applied with a uniform prior, yielding $B= 1+D\log(TD)$, or to $\prior$ a Dirac distribution $\delta_{\lawa}$, yielding a result for GRO with $B=1$.

Also remark that, from \cite[Theorem 3.1]{agrawalStoppingTimesPowerone2025}, the lower bound on sample size has a first order in $\frac{\log(1/\alpha)}{\KL(Q,P^\star)}$ with $P^\star$ the RIPr of $Q$ as $\alpha$ goes to 0, all other variables remaining constant. This matches our results, up to a constant factor in the case of GRO. Note that we can optimise the $\varepsilon$ chosen in the application \Cref{lem:GRO}; by making $\varepsilon$ small, we can make this constant factor arbitrarily close to 1.


\subsection{Fixed confidence test}\label{sect:PPOE}

We have the following result for the Fixed confidence test, direct consequence of Ville's inequality applied to both $E$ and $E'$:
\begin{theorem}\label{th:doubletest}
	The testing procedure for stopped e-sequential test described above has, at any time, a false positive rate smaller than $\alpha$, and a power higher than $1-\beta$. 
\end{theorem}

We then have, for the PPOE test, the following guarantee:

\begin{theorem}\label{th:auxfc}
	For any $Q=P_\lawa\in\Hyp_i$, for any $\gamma\in (0,1)$, the described PPOE test with prior $\prior$ satisfies $$\E_Q[\tau] \leq  T+\frac{2}{1-\exp\left[ -\frac{\gamma^2 \KL(Q,P^\star)}{\left(\log \frac{1}{\varepsilon}\right)^2 2 D}\right]}$$ for any $T$ that satisfies $$T \geq \frac{\log\frac{1}{\nu_i}+1+\log\left(\frac{1}{\prior(\mathcal{B}_{\infty}\left( \left( 1-\frac 1{TD}\right)\lawa, \frac{1}{TD}\right))}\right)}{(1-\gamma)\KL(Q,P^\star)} $$ where $P^\star$ is the RIPr of $Q$ onto $\Hyp_{1-i}$, and $\nu_i = \alpha$ if $i=1$ and $\beta$ otherwise.
\end{theorem}
The proof of this result can be found in \Cref{appsect:proofauxfc}.

For GRO (i.e. Dirac prior), this becomes \begin{align*} \E_Q[\tau] &\leq \frac{1+\log\frac{1}{\nu_i}}{\KL(Q,P^\star) (1-\gamma)} + \frac{2}{1-\exp\left[ -\frac{\gamma^2 \KL(Q,P^\star)}{\left(\log \frac{1}{\varepsilon}\right)^2 2 D}\right]}\end{align*} for any $\gamma\in (0,1)$. 
For POE with uniform prior, the condition on $T$ becomes $$T\geq \frac{\log\frac{1}{\nu_i}+1+D\log(TD)}{(1-\gamma)\KL(Q,P^\star)}.$$
In both cases, the first order in $\nu_i$ again matches the lower bound of \cite[Theorem 3.1]{agrawalStoppingTimesPowerone2025}.

Note that the bound of \Cref{lem:GRO} cannot be applied for parallel tests; indeed, this would require (from \Cref{lem:pointpoint}) that both $\Param_0$ and $\Param_1$ be convex for both GRO e-values to be likelihood ratios between two distributions with Dirac prior, which is never the case in our setting.

        
\section{Computing POE: a Frank-Wolfe algorithm for reverse information projection}
To compute POE, we need to solve a reverse information projection problem at each timestep to compute $W_{0,t}$ the projection of the posterior onto $\Hyp_0$. The generic method used so far (when no specific method such as \cite[Equation (11)]{turnerExactAnytimevalidConfidence2022} could be found) is Li's algorithm (see details in \Cref{sec:Li}). However, this algorithm is rather slow and, in high dimension, it is not computable, due to the very high complexity of the (non-linear) optimisation problem at each iteration. We give here an alternative way to do reverse information projection. 
For that, we look towards the Frank-Wolfe algorithm, introduced in \cite{frankAlgorithmQuadraticProgramming1956}. This algorithm aims to find $x$ that minimises a smooth function $f$ on a convex set, by making smaller and smaller steps in the directions $y_t$ that minimise the gradient $\nabla f(x_t)^\top y_t$. Here, the function we want to minimise is $P\mapsto \KL(P_{W_{1,t}},P)$. 
While this algorithm can be applied in multiple settings, we will apply the results of the finite-dimension case found in \cite[Section 3.3]{bubeckConvexOptimizationAlgorithms2015}. Indeed, one can note that, in our multi-Bernoulli setting, the support of all the $P_W$ are finite: $P_W$ is uniquely determined by the $2^D$ coefficients $P_W(x)$ for $x\in\{0,1\}^D$. Therefore, we can recast the problem in $\R^{2^D}$. The scalar product in $\R^{2^D}$ is \[ \langle P,Q \rangle = \sum_x P(x) Q(x) = \E_P[Q] . \] This manipulation can be conducted for any finite-support distribution.

For any function $f$ and $x_t$, the minimisation problem $\argmin_{y_t}\nabla f(x_t)^\top y_t$ is linear, we can restrict the search for the minimiser $y_t$ in the extremal points of our convex set. 
Since we want to find the $P_W$ minimising $\KL(P_{W_{1,t}},P_W)$, we need to know the extremal points of the set of $P_W$ where $W\in \Prior_0$.

\begin{lemma}
	The extremal points of $P_{\Prior'} = \{P_W,W\in \Prior'\}$ for $\Prior'$ the set of priors over some $\Param'\subseteq \Param$ are all in the set $\{P_{\delta_\param},\param\in \Param'\}$.
\end{lemma}

\begin{proof}
	By definition, an extremal point of $P_{\Prior'}$ is a point which cannot be written as a non-trivial convex combination (i.e. a convex combination with all weights non-zero) of points in $P_{\Prior'}$. Suppose that $P_W$ is an extremal point of $P_{\Prior'}$ for some $W$ that is not a Dirac distribution. Because $W$ is not a Dirac distribution, its support is not a singleton and we can partition $\supp(W)$ into two disjoint non-empty sets $A$ and $B$. We can then write 
	\[ P_W(\cdot) = P_{W \mid A}(\cdot) \P_W(A) + P_{W \mid B}(\cdot) \P_W(B),\]
	which means that $P_W$ is a convex combination of $P_{W \mid A}$ and $P_{W \mid B}$ and it is not a trivial convex combination because $A$ and $B$ are both non-empty. This is a contradiction with $P_W$ an extremal point of $P_{\Prior'}$.
\end{proof}

We will therefore need to take steps in the direction of those $P_{\delta_\param}$. To know in which direction we step, we need to know the gradient of the function to minimise. Simple computation yields the following lemma proved in \Cref{sec:proof_deriv_fw}.

\begin{lemma}\label{lem:deriv_fw}
	For any $W_0,W'\in\Prior_0$ and $W_1\in\Prior_1$, we have \[ \frac{\d}{\d t} \KL(P_{W_1},(1-t)P_{W_0} + t P_{W'})\mid_{t=0} = 1-\E_{P_{W'}} \left[ \frac{P_{W_1}}{P_{W_0}}\right] \] so that \[ \argmin_{P_{W'}} \nabla \KL (P_{W_1},P_{W_0}) ^\top P_{W'} =\argmin_{P_{W'}} -\E_{P_{W'}} \left[ \frac{P_{W_1}}{P_{W_0}}\right]  \] (where the gradient is taken over the second argument of the KL).
\end{lemma}


The Frank Wolfe algorithm applied to our problem (compute the $W_0\in \Prior_0$ minimising $\KL(P_{W_1},P_{W_0})$) thus becomes Algorithm~\ref{alg:FW}.

\begin{algorithm}[!ht]
	\caption{The Frank Wolfe algorithm}
	\label{alg:FW}
	\KwData{Initialisation $W_0$, distribution $P_{W_1}$}
	\For{$k=0,1,\dots$}{
		Compute the minimiser over the corners \[\param_k = \argmin_{\param\in\Param_0} -\E_{P_{\delta_\param}} \left[ \frac{P_{W_1}}{P_{W_0}}\right]\]  \\ \label{line:easymin}
		Update the current estimate \[W_0\gets \left( 1-\frac{2}{2+k}\right) W_0 + \frac{2}{2+k} \delta_{\param_k}\]
	}
\end{algorithm}

While we use this manipulation (using the finiteness of support of our distributions to recast the problem in finite dimension) to be able to apply Frank-Wolfe in the simple case of functions on $\R^{2^D}$ \citep{bubeckConvexOptimizationAlgorithms2015}, we hypothesise that it could be applied to more general measure sets. We keep this for future work.

In practice, the stopping criterion can be that the gap between $W_0$ at steps $k$ and $k+1$ is small enough, or (if it is easily computable) that the $\KL$ between $P_{W_1}$ and $P_{W_0}$ is not changing much anymore. However, for more theoretically guaranteed ways of stopping, we have to consider our specific setting.

\subsection{Implementation details}\label{sect:fwimpl}

The minimisation problem of the Frank Wolfe algorithm is easier to compute than Li's algorithm. Indeed, in Algorithm~\ref{alg:FW}, the minimisation problem is linear in $\delta_\param$ (the $P_{W_0}$ in the expectation does not depend on $\param$). Recall that the parameter $\param$ is the preference matrix $\PrM = (\pr{i,j})_{1\leq i,j\leq N}$; since the support of any $P_W$ is finite, we can thus also see $\E_{P_{\delta_\param}} \left[ \frac{P_{W_1}}{P_{W_0}}\right]$ as a polynomial of degree 1 in each of the $\pr{i,j}$, which makes the optimisation problem smooth and that we can solve with off-the-shelf optimisation algorithms, even potentially for non convex $\Param_0$. 

Furthermore, for a $\Param_0$ that is a finite reunion of convex sets $\Param_{0,1},\dots,\Param_{0,J}$, one can solve Line \ref{line:easymin} by solving it over each $\Param_{0,j}$ to get a $\param_j$, picking $j^\star \in \argmin_{j\in [J]} 1-\E_{P_{\delta_{\param_j}}}\left[ \frac{P_{W_1}}{P_{W_0}}\right]$, and using $\param_{j^\star}$.

In high dimension, merely computing the $\KL$ might be tricky. For $N=12$ candidates, $D=66$, so computing one expectation yields a sum over $2^{66}$ terms. However, if $\Param_0$ is convex, then, by \Cref{lem:pointpoint}, the RIPr of any $P_{W_1}$ onto $P_{\Param_0}$ has a prior with singleton support. Therefore, if we assume that all samples are taken independently (meaning that $X_{d}$ is independent from $X_{d'}$ for $d\neq d'$), then for a Dirac initialisation $W_0=\delta_{\param_0}$,\begin{align*}
	\E_{\delta_\param} \left[\frac{P_{W_1}}{P_{\delta_{\param_0}}}\right] &= \int \prod_{d=1}^D \frac{P_{\param^d}(X_d)}{P_{\param_0^d}(X_d)} P_{W_1}(X) \d X\\
	&= \int \prod_{d=1}^D \frac{P_{\param^d}(X_d)}{P_{\param_0^d}(X_d)} \left[ \int P_{\param'}(X) \d W_1(\param') \right] \d X \\
	&= \int \prod_{d=1}^D \frac{P_{\param^d}(X_d) P_{\param'^d}(X_d)}{P_{\param_0^d}(X_d)} \d X \d W_1(\param')
	\end{align*}
	and, using independence, \begin{align*}
		\E_{\delta_\param} \left[\frac{P_{W_1}}{P_{\delta_{\param_0}}}\right] &= \int \left( \prod_{d=1}^D \int \frac{P_{\param^d}(X_d) P_{\param'^d}(X_d)}{P_{\param_0^d}(X_d)} \d X_d\right) \d W_1(\param') \\
		&= \int \prod_{d=1}^D \E_{\Ber(\param^d)}\left[ \frac{\Ber(\param'^d)}{\Ber(\param_0^d)}\right] \d W_1(\param').
	\end{align*}
	Thus, for each $\param$, we can compute the objective function with linear complexity in $D$, and it remains to compute the integral over $W_1$ (which can be done by sampling over $W_1$).

While the Frank-Wolfe algorithm has convergence guarantees (see \cite[Theorem 3.8]{bubeckConvexOptimizationAlgorithms2015}), those are only valid for $\beta$-smooth functions. However, our $\KL$ function is not $\beta$-smooth: for $W$ arbitrarily close to the borders of $\Param$, the gradient of $\KL(P_{W_1},P_W)$ explodes. To solve that, we no longer consider $\Param$, but instead exclude the outer borders, by considering $\Param_\varepsilon = \left\{ \PrM, \forall i,j, \, \left| \log\frac{\pr{i,j}}{1- \pr{i,j}} \right| \leq \log(1/\varepsilon) \right\}$; our function is thus $\beta$-smooth with $\beta = \frac{1}{\varepsilon^{2D}}$ (and behaves way more smoothly far from the borders of $\Param_\varepsilon$), but we have lost sensitivity near the outer borders of $\Param$.

Note that the restriction $\param \in \Param_\varepsilon$ is strictly less restrictive than $\param_d \in (\varepsilon,1-\varepsilon)$ necessary in \Cref{th:POE}: in $\Param_\varepsilon$, the parameters satisfy $\param_d \in \left( \frac{\varepsilon}{1+\varepsilon},1-\frac{\varepsilon}{1+\varepsilon}\right)$.

We thus have a way to compute POE for a multitude of new settings, while the approach of \cite{turnerExactAnytimevalidConfidence2022} is limited to convex parameter sets.

\section{Experimental results in testing preferences}\label{sec:experimentation}

\subsection{Comparisons on synthetic examples}

In this section, we apply our tests and compare to state of the art methods. We first compare these methods in a simulated setting of testing whether one candidate is the winner in Borda, Condorcet and Schulze voting systems with 3 candidates. We fix $\Hyp_0$: ``candidate 1 is the winner'', and $\Hyp_1$: ``candidate 1 is not the winner. Then, we look at the effect of having access to prior data to inform $\pi$ beforehand, and on increasing the number of candidates.

\subsubsection{Power and Sample Complexity Comparison of Sequential Tests for Testing Condorcet, Borda and Schulze Winners}
In this section, we execute $M=1000$ repetitions of sequential tests with data simulated from a multivariate Bernoulli distribution $(X_{1,2},X_{1,3},X_{2,3})\sim \Ber(x) \otimes \Ber(0.6) \otimes \Ber(0.1)$ for a range of $x$ between $0.05$ and $0.95$. In order to keep the computation manageable, if the test did not conclude at time $t=100$, we stop testing and return ``Accept''. We plot the following two metrics: empirical probability to reject, and empirical average of sample size used to conclude. 

\begin{figure}[!ht]
	\begin{center}
		\includegraphics[width=1.1\textwidth]{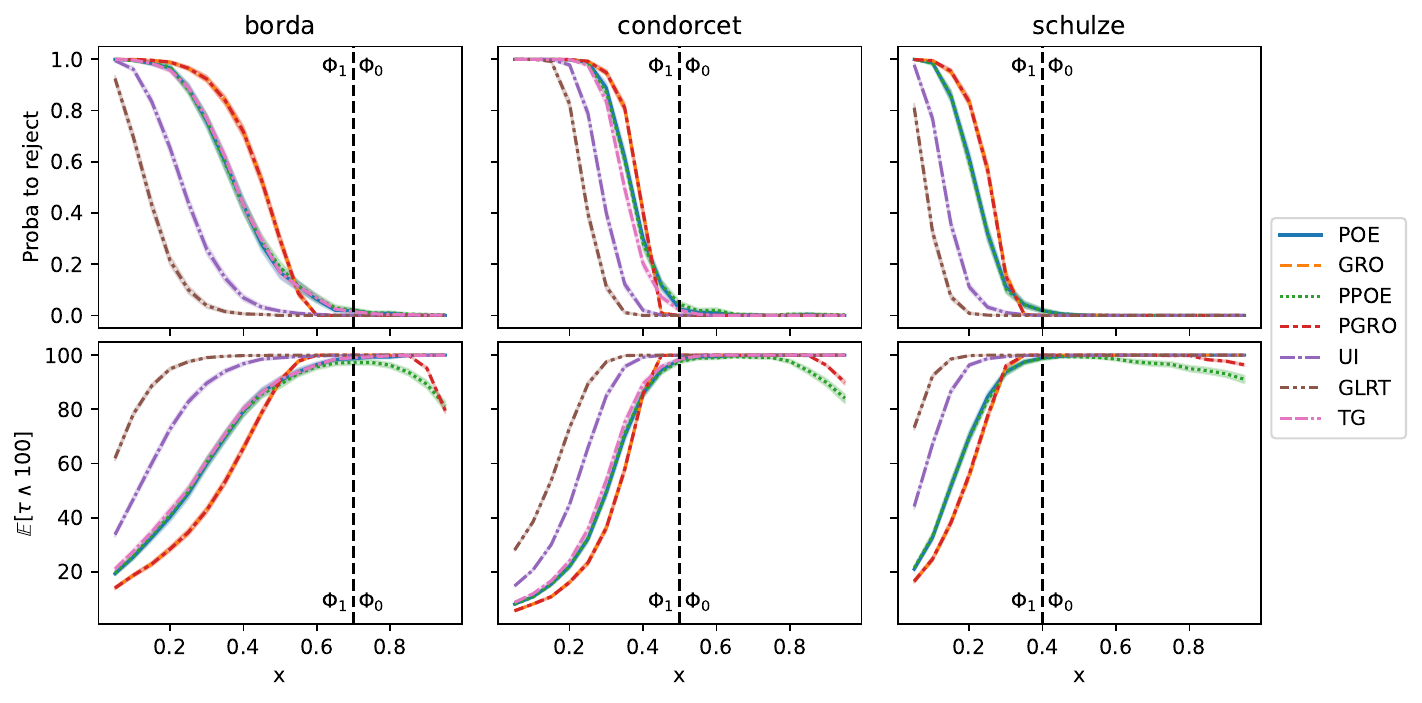}
	\end{center}
	\caption{Plot of the probability of deciding ``Reject $\Hyp_0$'' and the expected sample size of various sequential tests for data generated from $\Ber(x) \otimes \Ber(0.6) \otimes \Ber(0.1)$. $\Hyp_0$ corresponds to candidate $1$ being a winner in the associated voting system, depending on the column. The vertical separation line represents the separation between $\Hyp_0$ and $\Hyp_1$. 1000 repetitions.\label{fig:power_sample_complexity}}
\end{figure}

In the following, recall that a (here multivariate) Jeffrey prior is a product of Beta distributions with parameters $\left(\frac 1 2 , \frac 1 2\right)$.

We study the following sequential tests:
\begin{description}[itemsep=0.1pt,parsep=0.1pt,topsep=0.1pt]
	\item[POE:] our proposed Posterior Optimal E-Variable from \Cref{def:POE} with a Jeffrey prior restricted to $\Param_1$.
	\item[PPOE:] our proposed Parallel Posterior Optimal E-Variable from \Cref{sect:PPOE} with a Jeffrey prior restricted on their respective alternative hypothesis parameter set.
	\item[TG:] the Turner-Grunwald e-variable from inspired by \cite{turnerExactAnytimevalidConfidence2022}. It is close to our Posterior Optimal E-Variable, but with a Jeffrey prior on the whole hypercube $[0,1]^D$, which simplifies the computation and does not need our Frank-Wolfe algorithm for convex $\Param_0$.
	\item[GLRT:] the Generalized Likelihood Ratio test with the threshold from \cite{kaufmannMixtureMartingalesRevisited2021}.
	\item[GRO:] the Growth Rate Optimal e-variable from \cite{grunwaldSafeTesting2023}, with parameter the parameter of the real data generating distribution. GRO is then an oracle e-variable that depends on parameters that are not known to the learner. 
	\item[PGRO:] the parallel GRO e-variable from \Cref{sect:PPOE}, using the fact that it is equivalent to PPOE with a Dirac prior so we can use the Frank-Wolfe algorithm to compute it when the constraint set is not convex.
	\item[UI:] the Universal Inference e-variable from \cite{wassermanUniversalInference2020}, which can be seen as a sequential relaxation of GLRT to make it an e-variable. UI is known to be conservative, and strictly worse than TG in general because it computes the MLE on $\Hyp_0$ on the denominator, which is conservative compared to the mixture used in GRO and POE.
\end{description}

In Figure~\ref{fig:power_sample_complexity}, we see that, under $\Hyp_1$, GRO is most of the time better than the alternatives in terms of both power and sample size, except when close to the frontier of $H_1$ where the GRO e-variable is known to be $1$ (at the frontier, the GRO e-variable should be $1$ with probability $1$ because the $(x, 0.6,0.1)$ belong to $\Param_0$ and its projection is then trivial). We see that TG, POE and PPOE all give more or less the same results in $\Hyp_1$. The advantage of the parallel testing, which is possible only if we can project onto non-convex spaces, is clearly visible in the redescending sample complexity obtained in $\Hyp_0$ as we get further away from the border.

Note that GRO seems to have a brutal angle a bit before the frontier, with the probability of rejection suddenly hitting 0. This is due to the sample complexity being capped. Since GRO has less stochasticity than the other testing methods (it does not need to estimate the posterior distributions or the real parameter, since it is an oracle), it is way more sensitive to this cap. Other stochastic methods might have, for a same parameter $x$, more variety in the stopping time, yielding to a softer curve to 0 (for the rejection probability) or to 100 (for the sample complexity).

\subsubsection{Growth Rates of Sequential Tests for Testing Condorcet, Borda and Schulze Winner}
In this section, we plot the Growth Rate and the Posterior Growth Rate from \Cref{eq:gn}, which are the optimality criteria for respectively GRO and POE. We know that GRO should have the highest Growth rate, and POE the largest Posterior Growth rate.

\begin{figure}[h]
	\begin{center}
		\includegraphics[width=\textwidth]{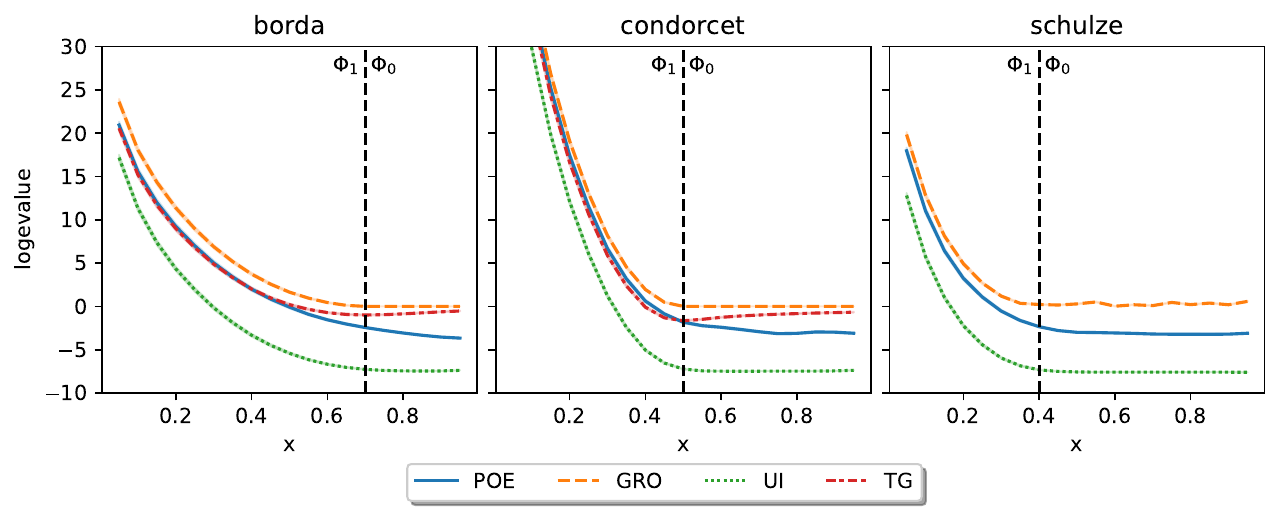}
	\end{center}
	\caption{Plot of estimated $\E_{\delta_\param}[\log(E)]$, with $\param=(x,0.6,0.1)$.
		\label{fig:plot_elog}}
\end{figure}

In Figure~\ref{fig:plot_elog} we can see that as expected GRO has the largest growth rate. Remark also that, by definition, POE has no incentive in optimising its growth rate for $Q \in \Hyp_0$, since it optimises the growth rate over its prior taken over $\Param_1$. This justifies its growth rate going down in $\Hyp_0$. However, the prior used by TG puts some weight in $\Param_0$, which could explain why its growth rate picks up in $\Hyp_0$.

As a sanity check, we also look at the Posterior Growth Rate $G_T$ and we compare POE with TG. See Figure~\ref{fig:plot_epilog} where we plot confidence intervals for $G_T(E,\pi,T)$ for $T = 100$ over 500 000 repetitions. Although POE has a higher $G_T(E,\pi,T)$, the value of  $G_T(E,\pi,T)$ for TG stays pretty close. 

\begin{figure}[h]
	\begin{center}
		\includegraphics[width=0.5\textwidth]{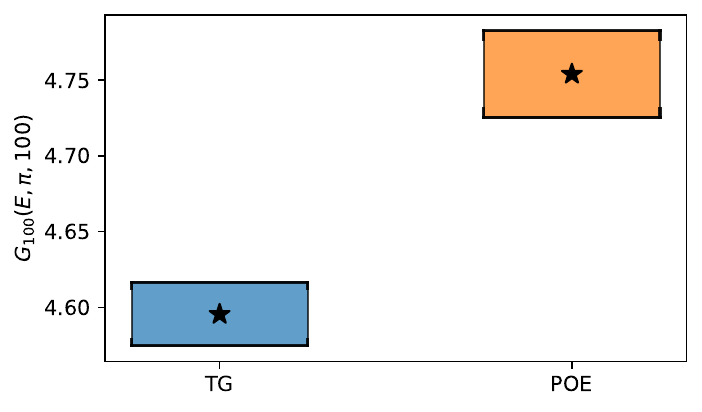}
	\end{center}
	\caption{Plot of confidence 95\% intervals for $G_{100}(E,\pi,100)$ using bootstrap confidence interval (500 000 repetitions) in Borda voting system with 3 candidates, with $\pi$ the Jeffrey prior restricted to $\Param_1$. \label{fig:plot_epilog}}
\end{figure}
\subsubsection{Effect of Prior Knowledge}
In this section, we consider the effect of having prior knowledge before doing the test; for example, if we have access to 10 or 50 data points collected beforehand from the sampling distribution. In this case, we can ``pre-train'' $\pi$ and use POE with $\pi$ the posterior distribution after having seen these data points.
\begin{figure}[!ht]
	\begin{center}
		\includegraphics[width=\textwidth]{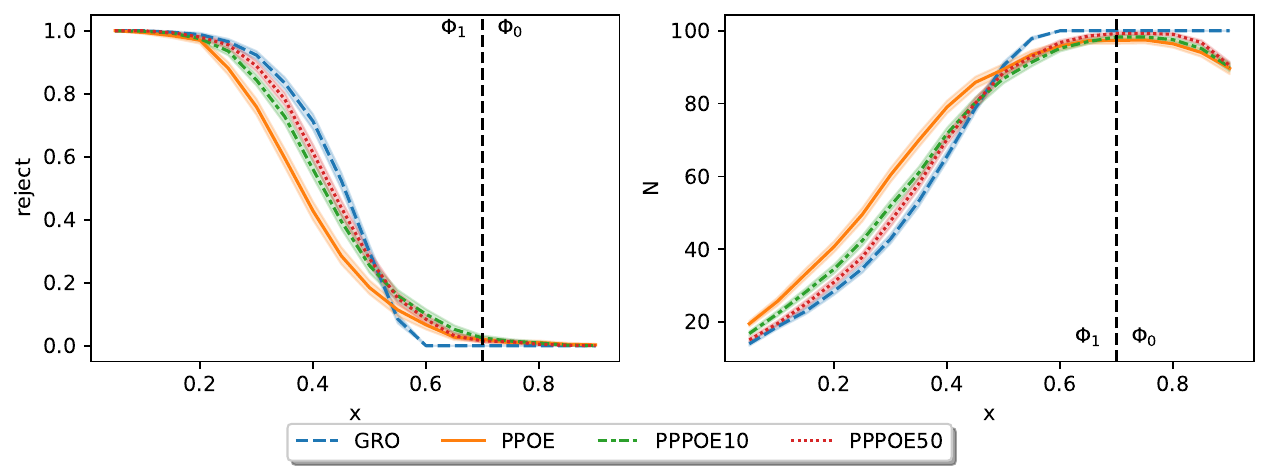}
	\end{center}
	\caption{Plot of the rejection probability (on the left) and the average stopping time (right) for GRO and for POE with 0, 10 and 50 prior data points in Borda voting system with 3 candidates.\label{fig:prior_data}}
\end{figure}
We plot in Figure~\ref{fig:prior_data} the rejection rate and the average stopping time of GRO and POE with more and more prior data. We see that the more prior data is injected into $\pi$, the closer POE is to GRO. This is pretty intuitive, because $\pi$ will concentrate more and more around the real sampling distribution parameter while GRO corresponds to POE with oracle prior $\pi$ set to a Dirac in the sampling distribution parameter.

\subsubsection{Effect of dimension on POE}\label{sec:xp_dim}

The computational time of TG is a lot smaller than POE due to the complexity of Frank Wolfe updates in particular, when $D$ becomes big. We investigate here if it is still worth it to use POE when $D$ grows. In Figure~\ref{fig:bar_dim}, we represent the rejection rate and average stopping time for POE and TG as the number of candidates $N$ grows (remember that $D=N(N-1)/2$).

\begin{figure}[h]
	\begin{center}
		\includegraphics[width=0.65\textwidth]{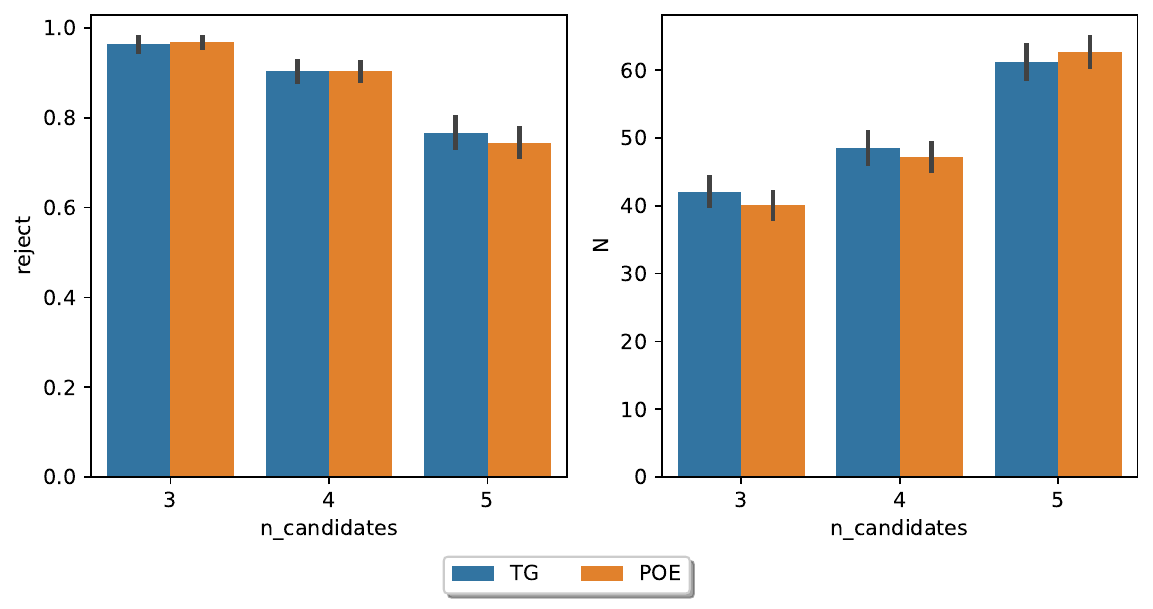}
	\end{center}
	\caption{Plot of the rejection probability (on the left) and the average stopping time (right) for increasing number of candidates in Borda voting system with parameter $(0.2, 0.6,0.6,\dots,0.6,0.1)$.\label{fig:bar_dim}}
\end{figure}
We see that TG and POE have more or less the same practical performances, which seems to confirm what was already witnessed in Figure~\ref{fig:power_sample_complexity}. This is in particular true because the size of $\Param_0$ gets smaller and smaller compared to $\Param$ as $D$ grows. 
Hence, if one is only interested in testing $\Hyp_0$ with a convex parameter set, TG seems to be a good choice, while $POE$ is needed as soon as $\Hyp_0$ is non-convex. In addition, even if in our setting $\Param_0$ is always small compared to $\Param \setminus \Param_0$, an open question is whether POE would empirically perform better than TG in problems where $\Param_0$ is big.

\subsection{``Voter autrement''}\label{sect:voterautrement}

In this section, we use the dataset and study from~\cite{delemazure_2024_10998451} to simulate a poll conducted according to votes cast in 2022 in French Presidential Election. This can be seen as an attempt to answer the questions: who would have been the winner if the Borda voting system had been used instead of the uninominal French voting system? Moreover, we are also interested in how many samples have to be collected (i.e. how many people have to be surveyed) to answer this question.

Taking inspiration from parallel GLRT from~\cite{kaufmannContributionsOptimalSolution2020}, we test simultaneously for each candidate $i$ the hypothesis $\Hyp_{0,i}$: "$i$ is a Borda winner" against $\Hyp_{1,i}$: "$i$ is not a Borda winner". Because in this setting, the parameter space $\Param_0$ is convex for each test, we use the TG estimator as its performances are similar to the ones of POE without the heavy computational cost (see Section~\ref{sec:xp_dim}).

The dataset from~\cite{delemazure_2024_10998451} provides pairwise votes coming from 2287 French voters, collected before and after the first round of the presidential elections. The responders also provided the ballot they actually cast in the official elections, and we use this actual ballot to reweight the sampling in order to match the official electoral results, in the same way that the original survey did. 
At each timestep, for each pair of candidates, we sample a duel between those two. For each candidate, we try to reject the hypothesis that they are a Borda winner, with false positive rate 5\%. Based on that, we plot in Figure~\ref{fig:evalues_voterautrement} the e-value corresponding to each candidate, depending on the number of iterations. When two candidates $i$ and $j$ are eliminated, it is no longer necessary to sample the preference $i$ versus $j$, thus reducing the number of samples per iteration. In the end, only one candidate, Yannick Jadot, remains; the dataset confirms that he is the theoretical Borda winner. Interestingly, his e-value remains equal to $1$ because the empirical distribution always lies within the hypothesis set corresponding to Jadot being a Borda winner. Consequently, no evidence accumulates against that hypothesis.

\begin{figure}[h]
	\begin{center}
		\includegraphics[width=\textwidth]{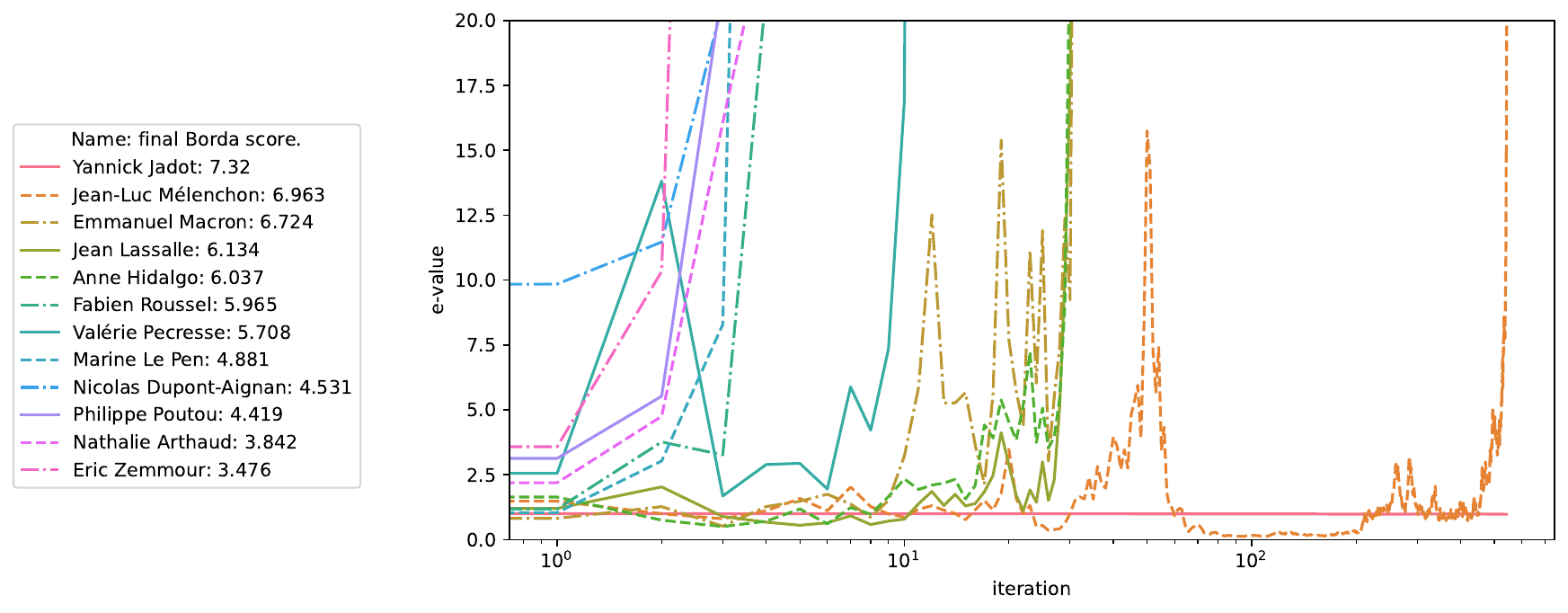}
	\end{center}
	\caption{Evolution of the sequential e-values for the candidates in the 2022 French presidential election dataset as the number of sampled pairwise comparisons increases. Candidates are progressively eliminated until only one, Jadot, remains.\label{fig:evalues_voterautrement}}
\end{figure}

\begin{figure}[h]
	\begin{center}
		\includegraphics[width=0.6\textwidth]{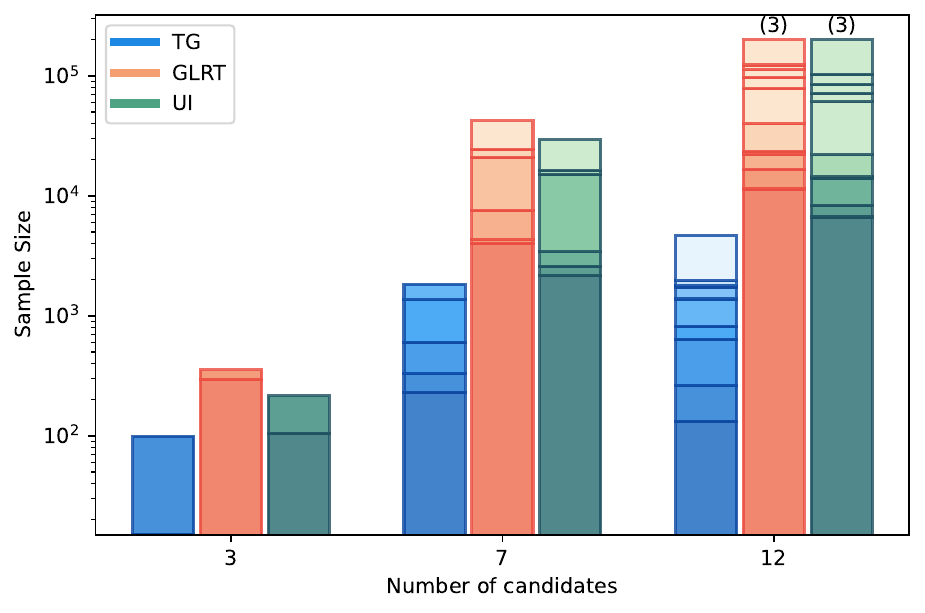}
	\end{center}
	\caption{Comparison of the sample complexity required to identify a single winner using three sequential testing methods; the generalised likelihood ratio test (GLRT), and two e-value-based tests: our proposition (TG) and universal inference (UI), for different numbers of candidates.\label{fig:samplesize_voterautrement}}
\end{figure}

In \Cref{fig:samplesize_voterautrement}, we plot the number of samples required to retain only one potential winner, comparing TG to two baselines: Universal Inference (see \cite{wassermanUniversalInference2020} Section 7), and the GLRT test (from \cite{kaufmannContributionsOptimalSolution2020} Section 3.2). We stop after collecting $2\cdot10^5$ samples, and -- if the test hasn't yet yielded a single winner -- we give the number of candidates still not rejected by the end of sample collection. Note that, while the number of samples seems to increase with the number of candidates, it is not always that straightforward: adding a very dominant candidate to a pool of equivalent candidates will make the test easier. Still, in those simulations, the e-value based test manages to conclude quicker than our two baselines, especially so as the instance gets more costly. For reference, official political surveys (IPSOS) conducted prior to the 2022 election used a total of $7,321$ voters, in contrast TG uses $171$ voters for a sample size of $4,721$ pairwise comparisons.

	\section{Conclusion and future work}

We have exhibited a new criterion for e-value optimality, posterior optimality, and have derived properties for it in \Cref{th:bayesian}. 
We have then given a new algorithm based on Frank Wolfe to compute the reverse information projection in our multivariate Bernoulli setting. Unlike the state of the art which can only deal with convex $\Param_0$, our approach even works in settings with non-convex parameter sets. Then, we have studied the guarantees of the oracle GRO e-value and of ours in the fixed-budget and in a fixed-confidence setting with parallel tests. Finally, we have conducted experiments in our preference-voting poll setting. We have studied and compared multiple testing procedures, their stopping time, their reject probability, on synthetic examples for three different winner definitions, and we have also implemented a test on real-life data, highlighting the feasibility of the test in real-life situations, even in high dimension.

Some questions remain open, however. 
First of all, the Frank Wolfe algorithm for RIPr is probably viable for a variety of settings, and not just our multivariate Bernoulli one. In which circumstances exactly is it a correct algorithm? What can we say about its convergence rate? Is the artificial $\varepsilon$ slack that we added a necessity, or can it be foregone somehow? 
Moreover, we wonder if it would be possible to find better ways to test for the Schulze winner; current approaches fail due to the number of convex components exploding with the dimension (as a reminder, projecting onto the subset of $\Param$ where candidate $i$ is a Schulze winner requires projecting onto $(D-1)!\times 2^D$ convex components). Efficient methods to tackle parameter sets with a high number of convex components would be a useful addition, and would allow for many more practical applications. 
Finally, we have seen that TG and POE have the same practical performance in settings where both are applicable. It remains an open question in which settings this resemblance is true: would a bigger $\Hyp_0$ cause POE to perform better, for example? When is the gap between the two appreciable? More work in the differences between TG and POE in practice is required.

	\section*{Acknowledgements}
	The authors would like to thank Rémy Degenne, Emilie Kaufmann and Peter Grünwald for their advice and discussions that helped us make this article what it is now.

        The authors would also like to acknowledge the Scool team for its great working environment as well as the ANR projects STRESS (ANR-25-CE23-6825) and FATE (ANR22-CE23-0016-01). 
	
	\bibliographystyle{apalike}
	\bibliography{bibli.bib}

@article{bubeckConvexOptimizationAlgorithms2015,
	title = {Convex {{Optimization}}: {{Algorithms}} and {{Complexity}}},
	shorttitle = {Convex {{Optimization}}},
	author = {Bubeck, S{\'e}bastien},
	year = 2015,
	month = nov,
	journal = {Foundations and Trends\textregistered{} in Machine Learning},
	volume = {8},
	number = {3-4},
	pages = {231--357},
	issn = {1935-8237},
	doi = {10.1561/2200000050},
	urldate = {2026-06-02}
}

@article{frankAlgorithmQuadraticProgramming1956,
	title = {An Algorithm for Quadratic Programming},
	author = {Frank, Marguerite and Wolfe, Philip},
	year = 1956,
	journal = {Naval Research Logistics Quarterly},
	volume = {3},
	number = {1-2},
	pages = {95--110},
	issn = {1931-9193},
	doi = {10.1002/nav.3800030109},
	urldate = {2026-06-02},
	copyright = {Copyright \copyright{} 1956 Wiley Periodicals, Inc., A Wiley Company}
}

@misc{agrawalStoppingTimesPowerone2025,
  title = {On {{Stopping Times}} of {{Power-one Sequential Tests}}: {{Tight Lower}} and {{Upper Bounds}}},
  shorttitle = {On {{Stopping Times}} of {{Power-one Sequential Tests}}},
  author = {Agrawal, Shubhada and Ramdas, Aaditya},
  year = 2025,
  month = apr,
  number = {arXiv:2504.19952},
  eprint = {2504.19952},
  primaryclass = {math.ST},
  publisher = {arXiv},
  doi = {10.48550/arXiv.2504.19952},
  urldate = {2026-06-09},
  archiveprefix = {arXiv}
}

@article{koolenLogoptimalAnytimevalidEvalues2022,
  title = {Log-Optimal Anytime-Valid {{E-values}}},
  author = {Koolen, Wouter M. and Gr{\"u}nwald, Peter},
  year = 2022,
  month = feb,
  journal = {International Journal of Approximate Reasoning},
  series = {Probability and {{Statistics}}: {{Foundations}} and {{History}}. {{In}} Honor of {{Glenn Shafer}}},
  volume = {141},
  pages = {69--82},
  issn = {0888-613X},
  doi = {10.1016/j.ijar.2021.09.010},
  urldate = {2026-05-26}
}

@book{cover1999elements,
  title={Elements of information theory},
  author={Cover, Thomas M},
  year={1999},
  publisher={John Wiley \& Sons}
}

@book{robertBayesianChoiceDecisionTheoretic2007,
  title = {The {{Bayesian Choice}}: {{From Decision-Theoretic Foundations}} to {{Computational Implementation}}},
  shorttitle = {The {{Bayesian Choice}}},
  author = {Robert, Christian P.},
  year = 2007,
  publisher = {Springer Verlag, New York},
  address = {New York, NY},
  isbn = {978-0-387-71598-8},
  langid = {english}
}

@misc{grunwaldNeymanPearsonEvaluesEnable2024,
  title = {Beyond {{Neyman-Pearson}}: E-Values Enable Hypothesis Testing with a Data-Driven Alpha},
  shorttitle = {Beyond {{Neyman-Pearson}}},
  author = {Gr{\"u}nwald, Peter},
  year = 2024,
  month = apr,
  number = {arXiv:2205.00901},
  eprint = {2205.00901},
  primaryclass = {stat},
  publisher = {arXiv},
  doi = {10.48550/arXiv.2205.00901},
  urldate = {2025-05-28},
  archiveprefix = {arXiv}
}

@techreport{kaufmannContributionsOptimalSolution2020,
	type = {Habilitation \`a Diriger Des Recherches},
	title = {Contributions to the {{Optimal Solution}} of {{Several Bandit Problems}}},
	author = {Kaufmann, Emilie},
	year = 2020,
	institution = {Universit\'e de Lille},
	urldate = {2026-05-26}
}

@dataset{delemazure_2024_10998451,
  author       = {Delemazure, Théo and
                  Bouveret, Sylvain},
  title        = {Voter Autrement 2022 - The Online Experiment ("Un
                   Autre Vote'')
                  },
  month        = apr,
  year         = 2024,
  publisher    = {Zenodo},
  doi          = {10.5281/zenodo.10998451},
  url          = {https://doi.org/10.5281/zenodo.10998451},
}

@misc{spertusSweeterSUITESupermartingale2022,
  title = {Sweeter than {{SUITE}}: {{Supermartingale Stratified Union-Intersection Tests}} of {{Elections}}},
  shorttitle = {Sweeter than {{SUITE}}},
  author = {Spertus, Jacob V. and Stark, Philip B.},
  year = 2022,
  month = jul,
  number = {arXiv:2207.03379},
  eprint = {2207.03379},
  primaryclass = {stat},
  publisher = {arXiv},
  doi = {10.48550/arXiv.2207.03379},
  urldate = {2024-12-14},
  archiveprefix = {arXiv}
}

@misc{waudby-smithRiLACSRiskLimitingAudits2021,
  title = {{{RiLACS}}: {{Risk-Limiting Audits}} via {{Confidence Sequences}}},
  shorttitle = {{{RiLACS}}},
  author = {{Waudby-Smith}, Ian and Stark, Philip B. and Ramdas, Aaditya},
  year = 2021,
  month = nov,
  number = {arXiv:2107.11323},
  eprint = {2107.11323},
  primaryclass = {stat},
  publisher = {arXiv},
  doi = {10.48550/arXiv.2107.11323},
  urldate = {2024-12-14},
  archiveprefix = {arXiv}
}

@article{ramdasHypothesisTestingEvalues2025,
  title = {Hypothesis {{Testing}} with {{E-values}}},
  author = {Ramdas, Aaditya and Wang, Ruodu},
  year = 2025,
  month = jul,
  journal = {Foundations and Trends in Statistics},
  volume = {1},
  number = {1-2},
  pages = {1--390},
  issn = {2978-4212},
  doi = {10.1561/3600000002},
  urldate = {2026-06-23}
}

@article{turnerGenericEvariablesExact2024,
  title = {Generic {{E-variables}} for Exact Sequential K{$<$}math{$><$}mi Is="true"{$>$}k{$<$}/Mi{$><$}/Math{$>$}-Sample Tests That Allow for Optional Stopping},
  author = {Turner, Rosanne J. and Ly, Alexander and Gr{\"u}nwald, Peter D.},
  year = 2024,
  month = may,
  journal = {Journal of Statistical Planning and Inference},
  volume = {230},
  pages = {106116},
  issn = {0378-3758},
  doi = {10.1016/j.jspi.2023.106116},
  urldate = {2026-06-03}
}

@misc{makurMinimaxHypothesisTesting2024,
  title = {Minimax {{Hypothesis Testing}} for the {{Bradley-Terry-Luce Model}}},
  author = {Makur, Anuran and Singh, Japneet},
  year = 2024,
  month = oct,
  number = {arXiv:2410.08360},
  eprint = {2410.08360},
  publisher = {arXiv},
  doi = {10.48550/arXiv.2410.08360},
  urldate = {2024-10-14},
  archiveprefix = {arXiv}
}

@misc{haoEValuesExponentialFamilies2025,
  title = {E-{{Values}} for {{Exponential Families}}: The {{General Case}}},
  shorttitle = {E-{{Values}} for {{Exponential Families}}},
  author = {Hao, Yunda and Gr{\"u}nwald, Peter},
  year = 2025,
  month = aug,
  number = {arXiv:2409.11134},
  eprint = {2409.11134},
  primaryclass = {stat.ME},
  publisher = {arXiv},
  doi = {10.48550/arXiv.2409.11134},
  urldate = {2026-06-22},
  archiveprefix = {arXiv}
}

@misc{arnoldOptimalEvaluesTesting2026,
  title = {Optimal E-Values for Testing the Mean of a Bounded Random Variable against a Composite Alternative},
  author = {Arnold, Sebastian and Clerico, Eugenio},
  year = 2026,
  month = jan,
  number = {arXiv:2601.11347},
  eprint = {2601.11347},
  primaryclass = {math},
  publisher = {arXiv},
  doi = {10.48550/arXiv.2601.11347},
  urldate = {2026-01-19},
  archiveprefix = {arXiv}
}

@misc{grunwaldOptimalEValuesExponential2025,
  title = {Optimal {{E-Values}} for {{Exponential Families}}: The {{Simple Case}}},
  shorttitle = {Optimal {{E-Values}} for {{Exponential Families}}},
  author = {Gr{\"u}nwald, Peter and Lardy, Tyron and Hao, Yunda and {Bar-Lev}, Shaul K. and de Jong, Martijn},
  year = 2025,
  month = apr,
  number = {arXiv:2404.19465},
  eprint = {2404.19465},
  primaryclass = {stat.ME},
  publisher = {arXiv},
  doi = {10.48550/arXiv.2404.19465},
  urldate = {2026-06-10},
  archiveprefix = {arXiv}
}

@article{agrawalOptimalBestarmIdentification2021,
  title = {Optimal Best-Arm Identification Methods for Tail-Risk Measures},
  author = {Agrawal, Shubhada and Koolen, Wouter M and Juneja, Sandeep},
  year = {2021},
  journal = {Advances in Neural Information Processing Systems},
  volume = {34},
  pages = {25578--25590},
}

@article{lardy2024reverse,
  title={Reverse information projections and optimal e-statistics},
  author={Lardy, Tyron and Gr{\"u}nwald, Peter and Harremo{\"e}s, Peter},
  journal={IEEE Transactions on Information Theory},
  year={2024},
  publisher={IEEE}
}

@misc{bengsPreferencebasedOnlineLearning2021,
  title = {Preference-Based {{Online Learning}} with {{Dueling Bandits}}: {{A Survey}}},
  shorttitle = {Preference-Based {{Online Learning}} with {{Dueling Bandits}}},
  author = {Bengs, Viktor and {Busa-Fekete}, Robert and {Mesaoudi-Paul}, Adil El and H{\"u}llermeier, Eyke},
  year = {2021},
  month = jul,
  number = {arXiv:1807.11398},
  eprint = {1807.11398},
  primaryclass = {cs, stat},
  publisher = {arXiv},
  doi = {10.48550/arXiv.1807.11398},
  urldate = {2024-05-05},
  archiveprefix = {arxiv}
}

@inproceedings{chenCombinatorialPureExploration2020,
  title = {Combinatorial {{Pure Exploration}} of {{Dueling Bandit}}},
  booktitle = {37th {{International Conference}} on {{Machine Learning}}},
  author = {Chen, Wei and Du, Yihan and Huang, Longbo and Zhao, Haoyu},
  year = {2020},
  month = jun,
  eprint = {2006.12772},
  primaryclass = {cs, stat},
  publisher = {PMLR},
  doi = {10.48550/arXiv.2006.12772},
  urldate = {2024-06-06},
  archiveprefix = {arxiv}
}

@misc{grunwaldSafeTesting2023,
  title = {Safe {{Testing}}},
  author = {Gr{\"u}nwald, Peter and {de Heide}, Rianne and Koolen, Wouter},
  year = {2023},
  month = mar,
  number = {arXiv:1906.07801},
  eprint = {1906.07801},
  primaryclass = {cs, math, stat},
  publisher = {arXiv},
  doi = {10.48550/arXiv.1906.07801},
  urldate = {2024-05-23},
  archiveprefix = {arxiv}
}

@inproceedings{haddenhorstIdentificationGeneralizedCondorcet2021,
  title = {Identification of the {{Generalized Condorcet Winner}} in {{Multi-dueling Bandits}}},
  booktitle = {Advances in {{Neural Information Processing Systems}}},
  author = {Haddenhorst, Bj{\"o}rn and Bengs, Viktor and H{\"u}llermeier, Eyke},
  year = {2021},
  volume = {34},
  pages = {25904--25916},
  publisher = {Curran Associates, Inc.},
  urldate = {2024-05-05}
}

@article{haddenhorstTestificationCondorcetWinners2021,
  title = {Testification of {{Condorcet Winners}} in {{Dueling Bandits}}},
  author = {Haddenhorst, Bj{\"o}rn and Bengs, Viktor and Brandt, Jasmin and H{\"u}llermeier, Eyke},
  year = {2021}
}

@misc{kaufmannMixtureMartingalesRevisited2021,
  title = {Mixture {{Martingales Revisited}} with {{Applications}} to {{Sequential Tests}} and {{Confidence Intervals}}},
  author = {Kaufmann, Emilie and Koolen, Wouter},
  year = {2021},
  month = dec,
  number = {arXiv:1811.11419},
  eprint = {1811.11419},
  primaryclass = {cs, stat},
  publisher = {arXiv},
  doi = {10.48550/arXiv.1811.11419},
  urldate = {2024-06-10},
  archiveprefix = {arxiv}
}

@phdthesis{liEstimationMixtureModels1999,
  title = {Estimation of {{Mixture Models}}},
  author = {Li, Jonathan},
  year = {1999},
  school = {Yale University}
}

@article{ramdasGameTheoreticStatisticsSafe2023,
  title = {Game-{{Theoretic Statistics}} and {{Safe Anytime-Valid Inference}}},
  author = {Ramdas, Aaditya and Gr{\"u}nwald, Peter and Vovk, Vladimir and Shafer, Glenn},
  year = {2023},
  month = nov,
  journal = {Statistical Science},
  volume = {38},
  number = {4},
  pages = {576--601},
  publisher = {Institute of Mathematical Statistics},
  issn = {0883-4237, 2168-8745},
  doi = {10.1214/23-STS894},
  urldate = {2024-05-30}
}

@misc{rastogiTwoSampleTestingRanked2020,
  title = {Two-{{Sample Testing}} on {{Ranked Preference Data}} and the {{Role}} of {{Modeling Assumptions}}},
  author = {Rastogi, Charvi and Balakrishnan, Sivaraman and Shah, Nihar B. and Singh, Aarti},
  year = {2020},
  month = nov,
  number = {arXiv:2006.11909},
  eprint = {2006.11909},
  primaryclass = {cs, math, stat},
  publisher = {arXiv},
  doi = {10.48550/arXiv.2006.11909},
  urldate = {2024-05-05},
  archiveprefix = {arxiv}
}

@misc{sahaPACInstanceOptimalSample2020,
  title = {From {{PAC}} to {{Instance-Optimal Sample Complexity}} in the {{Plackett-Luce Model}}},
  author = {Saha, Aadirupa and Gopalan, Aditya},
  year = {2020},
  month = feb,
  number = {arXiv:1903.00558},
  eprint = {1903.00558},
  primaryclass = {cs, stat},
  publisher = {arXiv},
  urldate = {2024-05-08},
  archiveprefix = {arxiv}
}

@inproceedings{sahaVersatileDuelingBandits2022,
  title = {Versatile {{Dueling Bandits}}: {{Best-of-both World Analyses}} for {{Learning}} from {{Relative Preferences}}},
  shorttitle = {Versatile {{Dueling Bandits}}},
  booktitle = {Proceedings of the 39th {{International Conference}} on {{Machine Learning}}},
  author = {Saha, Aadirupa and Gaillard, Pierre},
  year = {2022},
  month = jun,
  pages = {19011--19026},
  publisher = {PMLR},
  issn = {2640-3498},
  urldate = {2024-05-11}
}

@article{taplinStatisticalAnalysisPreference1997,
  title = {The {{Statistical Analysis}} of {{Preference Data}}},
  author = {Taplin, Ross H.},
  year = {1997},
  journal = {Journal of the Royal Statistical Society. Series C (Applied Statistics)},
  volume = {46},
  number = {4},
  eprint = {2986359},
  eprinttype = {jstor},
  pages = {493--512},
  publisher = {[Wiley, Royal Statistical Society]},
  issn = {0035-9254},
  doi = {10.1111/1467-9876.00086},
  urldate = {2024-05-05}
}

@misc{turnerExactAnytimevalidConfidence2022,
  title = {Exact {{Anytime-valid Confidence Intervals}} for {{Contingency Tables}} and {{Beyond}}},
  author = {Turner, Rosanne and Gr{\"u}nwald, Peter},
  year = {2022},
  month = jun,
  number = {arXiv:2203.09785},
  eprint = {2203.09785},
  primaryclass = {stat},
  publisher = {arXiv},
  doi = {10.48550/arXiv.2203.09785},
  urldate = {2024-06-14},
  archiveprefix = {arxiv}
}

@article{wassermanUniversalInference2020,
  title = {Universal {{Inference}}},
  author = {Wasserman, Larry and Ramdas, Aaditya and Balakrishnan, Sivaraman},
  year = {2020},
  month = jul,
  journal = {Proceedings of the National Academy of Sciences},
  volume = {117},
  number = {29},
  eprint = {1912.11436},
  primaryclass = {math, stat},
  pages = {16880--16890},
  issn = {0027-8424, 1091-6490},
  doi = {10.1073/pnas.1922664117},
  urldate = {2024-06-01},
  archiveprefix = {arxiv}
}

@misc{xiaOptimalStatisticalHypothesis2020,
  title = {Optimal {{Statistical Hypothesis Testing}} for {{Social Choice}}},
  author = {Xia, Lirong},
  year = {2020},
  month = jun,
  number = {arXiv:2006.11362},
  eprint = {2006.11362},
  primaryclass = {cs, math, stat},
  publisher = {arXiv},
  doi = {10.48550/arXiv.2006.11362},
  urldate = {2024-05-05},
  archiveprefix = {arxiv}
}

@article{schulzeNewMonotonicCloneindependent2011,
  title = {A New Monotonic, Clone-Independent, Reversal Symmetric, and Condorcet-Consistent Single-Winner Election Method},
  author = {Schulze, Markus},
  year = 2011,
  month = feb,
  journal = {Social Choice and Welfare},
  volume = {36},
  number = {2},
  pages = {267--303},
  issn = {1432-217X},
  doi = {10.1007/s00355-010-0475-4},
  urldate = {2025-12-05}
}

@book{ghoshSequentialEstimation1997,
	title = {Sequential {{Estimation}}},
	author = {Ghosh, Malay and Mukhopadhyay, Nitis and Sen, Pranab Kumar},
	year = 1997,
	publisher = {Wiley-Interscience},
	address = {New York},
	isbn = {978-0-471-81271-5}
}
        \appendix
        \section{Instant Runoff cannot be represented by a preference matrix}\label{sec:instant_runoff}

A good option to satisfy a lot of voting criteria is the instant-runoff voting system, which is a rather simple procedure. It is close to the French voting system, and it is used as is to elect members of the Australian House of Representatives, among many other uses.

In instant-runoff, all voters give a ranking of candidates (let us assume that each voter gives a ranking of every single candidate). Then, if one candidate is the first choice of the majority of voters, it is elected; otherwise, the candidate that was the first choice of the smallest amount of people is eliminated, and the procedure is iterated.
\begin{example}
	In the results of \Cref{fig:voting}, 20\% of voters preferred candidate A to B, and candidate B to C, yielding the voting order ABC. In \Cref{fig:votingresround1}, we look at the proportion of voters that place each candidate in the first position: 45\% of voters rank A first. No candidate has the majority; therefore, we eliminate the candidate that had the fewest first position votes, C, and we conduct another round.
\begin{figure}[h]
  \begin{center}
	\subfloat[A vote distribution\label{fig:voting}]{
		\centering
		\begin{tikzpicture}[
			scale=0.7]
			\pie
			[style={thin},color={tikzcolor1!50,tikzcolor1!50,tikzcolor2!50,tikzcolor2!50,tikzcolor3!50,tikzcolor3!50}]
			{
				20/ABC,
				25/ACB,
				5/BAC,
				25/BCA,
				15/CAB,
				10/CBA
			}
		\end{tikzpicture}
		}
	\subfloat[The first round\label{fig:votingresround1}]{
		\centering
		\begin{tikzpicture}[scale=0.7]
			
			\begin{axis} [ybar,xmin=0,xmax=4,ymin=0,ymax=0.65,xtick={1,2,3},xticklabels={A,B,C},
				extra y ticks={0.5},             
				extra y tick style={%
					grid=major,                  
					grid style={dashed, red},    
				},
				bar width=20pt,
				bar shift=0pt
				]
				\addplot[tikzcolor1, fill=tikzcolor1!50] coordinates {(1,0.45)};
				\addplot[tikzcolor2, fill=tikzcolor2!50] coordinates {(2,0.3)};
				\addplot[tikzcolor3, fill=tikzcolor3!50] coordinates {(3,0.25)};
			\end{axis}
		\end{tikzpicture}
              }
        \end{center}
	\end{figure}
	
	After eliminating C, the votes placing it in first place are redistributed according to the candidate placed second. In \Cref{fig:voting2}, A gains the CAB votes, B gains the CBA votes. With that, A has 60\% of votes as seen in \Cref{fig:votingresround2}, and A is the winner.
	\begin{figure}[h]
            \begin{center}
	\subfloat[Distribution post eliminating C\label{fig:voting2}]{
		\centering
		\begin{tikzpicture}[
			scale=0.7]
			\pie
			[style={thin},
			color={tikzcolor1!50, tikzcolor1!50, tikzcolor2!50, tikzcolor2!50, tikzcolor1!50, tikzcolor2!50}]
			{
				20/AB-,
				25/A-B,
				5/BA-,
				25/B-A,
				15/-AB,
				10/-BA
			}
		\end{tikzpicture}
		
	}
	\subfloat[The second round\label{fig:votingresround2}]{
		\centering
		\begin{tikzpicture}[scale=0.7]
			
			\begin{axis} [ybar,xmin=0,xmax=3,ymin=0,ymax=0.65,xtick={1,2},xticklabels={A,B},
				extra y ticks={0.5},             
				extra y tick style={%
					grid=major,                  
					grid style={dashed, red},    
				},
				bar width=20pt,
				bar shift=0pt
				]
				\addplot[tikzcolor1, fill=tikzcolor1!50] coordinates {(1,0.6)};
				\addplot[tikzcolor2, fill=tikzcolor2!50] coordinates {(2,0.4)};
			\end{axis}
		\end{tikzpicture}
		}
                \end{center}
	\end{figure}
\end{example}

We now give a second example showing that instant-runoff is not determined by the preference matrix.

\begin{example}\label{ex:runoff}
	For the voting distribution of \Cref{fig:voting}, we compute $\pr{AB}=0.6$, $\pr{AC}=0.5$ and $\pr{BC}=0.5$. Consider instead \Cref{fig:voting3}.
	
	\begin{figure}
            \begin{center}
          \subfloat[A vote distribution\label{fig:voting3}]{
		\centering
		\begin{tikzpicture}[
			scale=0.65]
			\pie
			[style={thin},color={tikzcolor1!50,tikzcolor1!50,tikzcolor2!50,tikzcolor2!50,tikzcolor3!50,tikzcolor3!50}]
			{
				20/ABC,
				5/ACB,
				25/BAC,
				5/BCA,
				35/CAB,
				10/CBA
			}
		\end{tikzpicture}
              }		
	\subfloat[The first round\label{fig:votingresbis}]{
		\centering
		\begin{tikzpicture}[scale=0.7]
			
			\begin{axis} [ybar,xmin=0,xmax=4,ymin=0,ymax=0.65,xtick={1,2,3},xticklabels={A,B,C},
				extra y ticks={0.5},             
				extra y tick style={%
					grid=major,                  
					grid style={dashed, red},    
				},
				bar width=20pt,
				bar shift=0pt
				]
				\addplot[tikzcolor1, fill=tikzcolor1!50] coordinates {(1,0.25)};
				\addplot[tikzcolor2, fill=tikzcolor2!50] coordinates {(2,0.3)};
				\addplot[tikzcolor3, fill=tikzcolor3!50] coordinates {(3,0.45)};
			\end{axis}
		\end{tikzpicture}
              }		
              \end{center}
              \end{figure}
	We compute that, in this second voting distribution, $\pr{AB}=0.6$, $\pr{AC}=0.5$ and $\pr{BC}=0.5$, and therefore we have the same preference matrix $\PrM$. However, \Cref{fig:votingresbis} shows that candidate A gets eliminated in the first round, and therefore cannot be the winner.
\end{example}

\newpage
\section{Li's algorithm for Reverse Information Projection}\label{sec:Li}
We present here the common algorithm of computing RIPr initially defined in \cite{liEstimationMixtureModels1999}, explained in a more modern version in \cite[Algorithm 1]{lardy2024reverse}. We restate the algorithm here for completeness:

\begin{algorithm}[!ht]
	\caption{Greedy approximation of the Reverse Information Projection}
	\label{alg:RIPli}
		Let $Q_1=Q_{\param_1}$ where $\param_1 = \argmin_{\param\in\Param} \KL(P,Q_\param)$\\
		\For{$k=2,3,\dots$}{
			Let $\alpha_k = \frac 2{k+1}$ \\
			Compute $\param_k = \argmin_{\param\in\Param} \KL(P,(1-\alpha_k)Q_{k-1}+\alpha_k Q_\param)$ \label{line:minKL} \\
			Let $Q_k = (1-\alpha_k)Q_{k-1} +\alpha_k Q_{\param_k}$
		}
\end{algorithm}

This algorithm converges towards a distribution $Q$ that minimises $\KL(P,Q)$. 
The hard part here is Line~\ref{line:minKL}. In high dimension, or in non convex spaces, minimising this quantity can be tricky - and, in our case, it is tricky.

        \section{Proof of Lemma \ref{lem:pointpoint}}\label{sec:proof_pointpoint}
	Let $\phi^*$ be such that 
	$$\KL(P_\psi^{\otimes T}, P_{\phi^*}^{\otimes T}) = \min_{\phi \in {\Param'}} \KL(P_\psi^{\otimes T}, P_\phi^{\otimes T}) $$
	which exists because 
	$$ \phi \mapsto \KL(P_\psi^{\otimes T}, P_\phi^{\otimes T}) = T \sum_{d=1}^D \kl\left(\psi_d, \phi_d\right)$$
	is a continuous function on $(0,1)^d$, coercive near the boundaries of $[0,1]^d$ and ${\Param'}$ is closed. Because $\phi^*$ is a minimum and ${\Param'}$ is convex, we have for any $\phi' \in {\Param'}$,
	$$\left(\nabla_\phi\sum_{d=1}^D \kl\left(\psi_d,\phi_d^*\right)\right)^T(\phi'-\phi^*) \ge 0.$$
	By direct computation, we have
	\begin{align*}
		\frac{\d}{\d \phi_d}\kl\left(\psi_d, \phi_d'\right) = \frac{\d}{\d \phi_d} \left(\psi_d \log \frac{\psi_d}{\phi_d} + (1-\psi_d)\log\frac{1-\psi_d}{1-\phi_d}\right)= -\frac{\psi_d }{\phi_d} + \frac{1-\psi_d}{1-\phi_d}
	\end{align*}
	Hence, 
	\begin{align*}
		\left(\nabla_\phi\sum_{d=1}^D \kl\left(\psi_d,\phi_d^*\right)\right)^T(\phi'-\phi^*) &= \sum_{d=1}^D (\phi_d'-\phi^*_d)\left(-\frac{\psi_d }{\phi_d^*} + \frac{1-\psi_d}{1-\phi_d^*}\right) \ge 0,
	\end{align*} 
	which we can rewrite
	\begin{align}\label{eq:evalue_gro_proof_deriv}
		0&\le \sum_{d=1}^D (\phi_d^*-\phi_d') \left(\frac{\psi_d}{\phi_d^*} - \frac{1-\psi_d}{1-\phi_d^*}\right)\nonumber \\
		&=\sum_{d=1}^D \phi_d^* \frac{\psi_d}{\phi_d^*}+ (1-\phi_d^*) \frac{1-\psi_d}{1-\phi_d^*}-\sum_{d=1}^D \phi_d' \frac{\psi_d}{\phi_d^*}+ (1-\phi_d') \frac{1-\psi_d}{1-\phi_d^*}\nonumber\\
		&=D -\sum_{d=1}^D \phi_d' \frac{\psi_d}{\phi_d^*}+ (1-\phi_d') \frac{1-\psi_d}{1-\phi_d^*}
	\end{align}
	Now, let us consider the random variable $E= \frac{\d P_\psi^{\otimes T}}{\d P_{\phi^*}^{\otimes T}}$. We want to prove that this is an e-variable.
	We have for any $\phi' \in {\Param'}$, 
	\begin{align*}
		\E_{P_{\phi'}^{\otimes T}}\left[E\right] &= \E_{P_{\phi'}^{\otimes T}}\left[\prod_{t=1}^T \prod_{d=1}^D \frac{\psi_d^{X_d^t}(1-\psi_d)^{1-X_d^t}}{(\phi_d^*)^{X_d^t}(1-\phi_d^*)^{1-X_d^t}}\right]\\
		&= \prod_{d=1}^D \left(\phi_d'\frac{\psi_d}{\phi_d^*} + (1- \phi_d')\frac{1-\psi_d}{1-\phi_d^*}\right)^T \\
		&= \exp\left(T\sum_{d=1}^D \log\left(\phi_d'\frac{\psi_d}{\phi_d^*} + (1- \phi_d')\frac{1-\psi_d}{1-\phi_d^*}\right)\right)
	\end{align*}
	By concavity of the logarithm function, we have
	$$\frac{1}{D}\sum_{d=1}^D \log\left(\phi_d'\frac{\psi_d}{\phi_d^*} + (1- \phi_d')\frac{1-\psi_d}{1-\phi_d^*}\right) \le \log\left(\frac{1}{D}\sum_{d=1}^D\left(\phi_d'\frac{\psi_d}{\phi_d^*} + (1- \phi_d')\frac{1-\psi_d}{1-\phi_d^*}\right)\right) $$
	and then, by Equation~\eqref{eq:evalue_gro_proof_deriv}, we get
	$$\frac{1}{D}\sum_{d=1}^D \log\left(\phi_d'\frac{\psi_d}{\phi_d^*} + (1- \phi_d')\frac{1-\psi_d}{1-\phi_d^*}\right) \le \log(1) = 0 .$$
	Hence, 
	$$\E_{P_{\phi'}^{\otimes T}}\left[E\right] \le 1 .$$
	Hence, $E$ is an e-variable. Because it is of the simple form of a likelihood ratio, from \Cref{th:gro} we know by unicity of the GRO as the only e-variable of the form of a likelihood ratio, that $E$ is the GRO e-variable. 
	
        \section{Proof of Theorem \ref{th:bayesian}}\label{sec:proof_bayesian}

First of all, let us make sense out of $W_{1,t}$. It is in fact the posterior distribution for the parameter at time $t$:
\begin{lemma}\label{lem:w1post}
	For any real function $f$ and integer $t$,
	\begin{align*}
		&\E_{\param\sim \pi}\left[ \E_{X^{:t} \sim P_\param^{\otimes t}} \left[ f(X^1,\dots,X^t)\right] \right] \\
		&\hspace{2em} =  \E_{\forall t'\leq t-1, X^{t'} \sim P_{W_{1,t'}}} \left[ \E_{X^t \sim P_{W_{1,t}}} \left[ f(X^1,\dots,X^t)\mid X^{:t-1}\right] \right] \\
		&\hspace{2em} =  \E_{\forall t'\leq t, X^{t'} \sim P_{W_{1,t'}}} \left[ f(X^1,\dots,X^t)\right] 
	\end{align*}
\end{lemma}

\begin{proof}
	\begin{align*}
		\prod_{t'=1}^t P_{W_{1,t'}}(X^t \mid X^{:t'-1})
		&= \prod_{t'=1}^t \frac{\int \prod_{i=1}^{t'} P_\param(X^i) \prior(\param) \d\param }{\int \prod_{i=1}^{t'-1} P_\param(X^i) \prior(\param) \d\param } \\
		&= 	\int \prod_{i=1}^t P_\param(X^i) \prior(\param) \d \param ,
	\end{align*}
	so that
	\begin{align*}
		\E_{\param\sim \pi}\left[ \E_{X^{:t} \sim P_\param^{\otimes t}} \left[ f(X^1,\dots,X^n)\right] \right] 
		& = \int \prior(\param) \prod_{t'=1}^t P_\param (X^i) f(X^{:t}) \d \param \d X^{:t} \\
		& = \int \prod_{t'=1}^t P_{W_{1,t'}}(X^t) f(X^{:t}) \d X^{:t} ,
	\end{align*} and we have the result.
\end{proof}

\begin{proof}[Proof of \Cref{th:bayesian}]
First, let us check that our candidate $(E_t)$ is a test supermartingale. 
By definition of $W_{0,t}$, we know that for any $W\in \Prior_0$, {\small\[ \KL \left(\int P_{\param}(X^t) \d W_{1,t}(\param),\int P_{\param}(X^t) \d W_{0,t}(\param)\right) \leq \KL \left(\int P_{\param}(X^t) \d W_{1,t}(\param),\int P_{\param}(X^t) \d W(\param)\right)\] } or, in other words, for any $W\in\Prior_0$ and for $\lambda=0$,
\[
\frac{\d}{\d \lambda} \left( \KL \left(\int P_{\param}(X^t) \d W_{1,t}(\param),\int P_{\param}(X^t) \d ((1-\lambda)W_{0,t}+\lambda W)(\param)\right)\right) \geq 0
\]
Recalling that $\KL(P,Q) = \E_P\left( \log \frac{P}{Q}\right)$ and deriving gives us, for any $W\in \Prior_0$,
\begin{align*} 1 & \geq \E_{X^t \sim P_{W_{1,t}}}\left[\frac{\int P_{\param}(X^t) \d W(\param)}{\int P_{\param}(X^t) \d W_{0,t}(\param)}\right] \\
	&= \E_{X^t \sim P_{W}}\left[\frac{\int P_{\param}(X^t) \d W_{1,t}(\param)}{\int P_{\param}(X^t) \d W_{0,t}(\param)}\right] .\end{align*}
Hence, for any $W\in \Prior_0$,
$$\E_{X^t\sim P_W}\left[\frac{P_{W_{1,t}}(X^t)}{P_{W_{0,t}}(X^t)}\mid X^{:t-1}\right] =\E_{X^t\sim P_{W}}\left[\frac{\int P_{\param}(X^t) \d W_{1,t}(\param;X^{:t-1})}{\int P_{\param}(X^t) \d W_{0,t}(\param;X^{:t-1})} \mid X^{:t-1}\right] \le1.$$
Then, we have by tower property that for any $t \in \N$,
\begin{align*}
	\E_{P_{W}^{\otimes t}}\left[\prod_{t'=1}^t \frac{P_{W_{1,t'}}(X^{t'})}{P_{W_{0,t'}}(X^{t'})} \mid X^{:t-1}\right]&=\prod_{t'=1}^{t-1} \frac{P_{W_{1,t'}}(X^{t'})}{P_{W_{0,t'}}(X^{t'})} \E_{ P_{W}}\left[\frac{P_{W_{1,t}}(X^t)}{P_{W_{0,t}}(X^t)} \mid X^{:t-1}\right]\\
	&\le \prod_{t'=1}^{t-1} \frac{P_{W_{1,t'}}(X^{t'})}{P_{W_{0,t'}}(X^{t'})}
\end{align*}
Hence, $E_t(X^{:t})$ is a supermartingale with (in particular) $\E_{P}[E_0]\le 1$ for any $P\in\Hyp_0$, hence it is a test supermartingale.

Now, we show log-optimality. Let 
$E_t'$ be a test supermartingale under $\Hyp_0$. Define 
$$M_t = \log(E_t') - \log(E_t)$$
We have that 
\begin{align*}
	&\E_{X^t  \sim P_\pi}\left[M_t \mid  X^{:t-1}\right] \\
	&\hspace{2em} =  \E_{X^t \sim P_\pi}\left[\log\left(\frac{E_t'}{E_{t-1}'}\right) - \log\left(\frac{P_{W_{1,t'}}(X^{t'})}{P_{W_{0,t'}}(X^{t'})}\right) \mid  X^{:t-1}\right] + M_{t-1}
\end{align*}
Now, remark that for any $P \in  \Hyp_0 $, we have \[\E_{X^t\sim P}\left[\frac{E_t'}{E_{t-1}'} \mid X^{:t-1} \right] =\frac{ \E_{X^t\sim P}\left[E_t' \mid X^{:t-1} \right]}{E_{t-1}'} \le 1\] since $E_t'$ is a supermartingale. Hence $\frac{E_t'}{E_{t-1}'}$ is an e-value on $\Hyp_0$ conditionally on $X^{:t-1}=x^{:t-1}$ for any realisation $x^{:t-1}$. By \Cref{th:gro} applied to the law
$$P_{W_{1,t}}(X^t)=\frac{\int  P_{\param}(X^t)\prod_{t'=1}^{t-1} P_{\param}(x_{t'}) \pi(\param) \d \param}{\int \prod_{t'=1}^{t-1} P_{\param}(x_{t'}) \pi(\param) \d \param} = P_\pi (X^t\mid X^{:t-1} = x^{:t-1})$$ with $\Hyp_1' = \{ P_{W_{1,t}} \}$,  
we have that 
\begin{align*}
	\E_{P_{W_{1,t}}}\left[\log\left(\frac{E_t'}{E_{t-1}'}\right) \right] &\le \E_{P_{W_{1,t}}}\left[\log\left(\frac{P_{W_{1,t}}( X^t)}{P_{W_{0,t}}(X^t)}\right)\right] 
\end{align*} 
Hence, 
\begin{equation} \label{eq:poe1} \E_{X^t\sim P_\pi}\left[M_t \mid  X^{:t-1} = x^{:t-1}\right]\le M_{t-1} .\end{equation}

If we define $Y_1 \sim \int P_\prior$, $Y_2\mid Y_1 \sim P_{W_{1,2}}$,\dots, $Y_t\mid Y^{:t-1}  \sim P_{W_{1,t}}$,\dots and $\filtr_t=\sigma(Y^{:t})$ the associated filtration, we have from \Cref{lem:w1post} that $G_t(E,\prior,\tau) = \E_{Y^{:t}}\left[ \log E(Y^{:t})\right] $ and \Cref{eq:poe1} can be rewritten \[\E [M_t\mid \filtr_{t-1}]\leq M_{t-1},\] so that $(M_t)$ is a supermartingale with respect to $(\filtr_t)$. Remark that, for any stopping time $\tau$, we have \begin{align*} 
	\E[M_{t\wedge \tau} \mid \filtr_{t-1}] &= \E[M_{t-1\wedge\tau} + (M_t - M_{t-1})\ind\left\{ \tau\geq t\right\} \mid \filtr_{t-1}] \\
	&= M_{t-1\wedge\tau} +\E[M_t-M_{t-1}\mid \filtr_{t-1}] \ind\left\{\tau\geq t\right\} \\ & \leq M_{t-1\wedge\tau}.
	\end{align*}
Therefore, $(M_{t\wedge\tau})$ is also a supermartingale with respect to $(\filtr_t)$. By the tower property, \begin{align*}
	\E[M_{t\wedge\tau}] &= \E\big[ \E[M_{t\wedge\tau}\mid \filtr_{t-1}]\big] \\
	&\leq \E\big[ \E[ M_{t-1\wedge\tau}\mid \filtr_{t-2}]\big]\end{align*}
and, by induction, $\E[M_{t\wedge\tau}] \leq \E[M_{0\wedge \tau}] =0$. By definition of $M_t$, this means \[G_t(E',\prior,\tau)\leq G_t(E,\prior,\tau)\] which means that $E$ is optimal.
%
\end{proof}

\section{Proof of concentrations for GRO}

\subsection{Proof of Lemma \ref{lem:GRO}}\label{appsect:proofsGro}

Set $Y_t = \sum_{d\in [D]} \log\frac{\lawa_d^{X_d^t}(1-\lawa_d)^{1-X_d^t}}{\lawb_d^{X_d^t}(1-\lawb_d)^{1-X_d^t}}$. As the $Y_t$'s are i.i.d., we have that for all $\lambda$ such that $\E_\Lawa[\exp(\lambda Y)]<\infty$ (which is any $\lambda$ in our setting), through the fundamental identity of sequential analysis~\cite[Corollary 2.4.1]{ghoshSequentialEstimation1997},
\begin{align}\label{eq:fundaidentitysequential}
	\E_\Lawa\left[ \exp\left(\lambda\sum_{t=1}^\tau Y_t\right) \E[\exp(\lambda Y)]^{-\tau}\right] &=1
\end{align}

Moreover, \begin{align*}
	\E_\Lawa[\exp(\lambda Y)] &= \E_\Lawa\left[ \exp\left( \lambda\sum_{d\in [D]} \log\frac{\lawa_d^{X_d}(1-\lawa_d)^{1-X_d}}{\lawb_d^{X_d}(1-\lawb_d)^{1-X_d}}\right)\right] \\
	&= \prod_{d\in [D]} \E_\Lawa\left[ \exp\left( \lambda \log\frac{\lawa_d^{X_d}(1-\lawa_d)^{1-X_d}}{\lawb_d^{X_d}(1-\lawb_d)^{1-X_d}}\right)\right]\end{align*} by the independence of the dimensions. Then, \begin{align*}
	\E_\Lawa[\exp(\lambda Y)]&= \prod_{d\in [D]} \E_\Lawa \left[ \exp\left( \lambda X_d \log\frac{\lawa_d(1-\lawb_d)}{\lawb_d(1-\lawa_d)} + \lambda\log\frac{1-\lawa_d}{1-\lawb_d}\right)\right]\\
	&= \prod_{d\in [D]} \exp\left( \lambda\log\frac{1-\lawa_d}{1-\lawb_d}\right) \E_\Lawa \left[ \exp\left(\lambda X_d \log\frac{\lawa_d(1-\lawb_d)}{\lawb_d(1-\lawa_d)}\right)\right]\\
	&\overset{\eqnum{1}}{\leq} \prod_{d\in [D]} \exp\left( \lambda\log\frac{1-\lawa_d}{1-\lawb_d}\right) \exp\left( \lambda\lawa_d \log\frac{\lawa_d(1-\lawb_d)}{\lawb_d(1-\lawa_d)} + \frac{\lambda^2}{8}\left(\log\frac{\lawa_d(1-\lawb_d)}{\lawb_d(1-\lawa_d)}\right)^2\right) \\
	&= \prod_{d\in [D]} \exp\left( \lambda\kl(\lawa_d,\lawb_d) + \frac{\lambda^2\sigma_d^2}{8}\right)
\end{align*}
where \texteqnum{1} is Hoeffding's lemma.
%
%
Therefore, using \Cref{eq:fundaidentitysequential}, we have by Markov's inequality for all $\lambda$ that, \begin{align*}
	\P_\Lawa \left( \frac{\exp (\lambda\sum_{t=1}^\tau Y_t)}{\prod_{d\in [D]} \exp\left( \lambda\kl(\lawa_d,\lawb_d)+\frac{\lambda^2\sigma_d^2}{8}\right)^\tau}\geq \frac{1}{\nu}\right) &\leq \nu
\end{align*}
and taking the log, we get
$$\P_\Lawa \left( \lambda\sum_{t=1}^\tau Y_t \geq \log \frac{\prod_{d\in [D]} \exp\left( \lambda\kl(\lawa_d,\lawb_d)+\frac{\lambda^2\sigma_d^2}{8}\right)^\tau}{\nu}\right)\leq \nu.$$
Using the fact that $\sum_{t=1}^\tau Y_t = \log\prod_{t=1}^\tau \frac{d\Lawa}{d\Lawb^\star}(X^t)$, where $\prod_{t=1}^{\tau-1} \frac{d\Lawa}{d\Lawb^\star}(X^t) <\frac{1}{\alpha}$ 
by definition of $\tau$ as the smallest crossing time, we have with $\lambda<0$ that, 
\begin{align*}
	\lambda\log\frac{1}{\alpha} +\lambda\sum_{d\in [D]} \max\left\{  \log\frac{1-\lawa_d}{1-\lawb_d}, \log\frac{\lawa_d}{\lawb_d}\right\} &\leq \log\frac{\prod_{d\in [D]} \exp\left( \lambda\kl(\lawa_d,\lawb_d)+\frac{\lambda^2\sigma_d^2}{8}\right)^\tau}{\nu}.
\end{align*}
Hence, rewritting with the notation, with probability greater than $1-\nu$, we have

\begin{align*}
\lambda\log\frac{1}{\alpha} +\lambda m( \lawa,\lawb) &\leq \log\prod_{d\in [D]} \exp\left( \lambda\kl(\lawa_d,\lawb_d)+\frac{\lambda^2\sigma_d^2}{8}\right)^\tau +\log \frac{1}{\nu}\\
 &= \tau \sum_{d\in [D]} \left( \lambda\kl(\lawa_d,\lawb_d)+\frac{\lambda^2\sigma_d^2}{8}\right),
\end{align*} 
so that, for any $\lambda$ such that $  -\frac{8\sum_{d} \kl(\lawa_d,\lawb_d)}{\sum_{d\in [D]} \sigma_d^2}\le \lambda\le 0$, we have $\sum_{d} \left( \lambda\kl(\lawa_d,\lawb_d)+\frac{\lambda^2\sigma_d^2}{8}\right) <0$ and
\begin{align*}
	\P_\Lawa\left(\tau \le  \frac{\lambda\log\frac{1}{\alpha} +\lambda m( \lawa,\lawb)+\log \nu}{\sum_{d\in [D]} \left( \lambda\kl(\lawa_d,\lawb_d)+\frac{\lambda^2\sigma_d^2}{8}\right)}  \right) &\ge 1- \nu
\end{align*} 

Now take $ \lambda = -\varepsilon \frac{8\sum_{d\in [D]} \kl(\lawa_d,\lawb_d)}{\sum_{d\in [D]} \sigma_d^2}$ for $\varepsilon\in (0,1)$, we get with probability larger than $1-\nu$,
\begin{align*}\label{eq:power}
	\tau \le \frac{\log\frac{1}{\alpha} +m( \lawa,\lawb)}{(1-\varepsilon)\sum_{d\in [D]}  \kl(\lawa_d,\lawb_d)}+\frac{\log (1/\nu) \sum_{d\in [D]} \sigma_d^2}{8\varepsilon(1-\varepsilon)\left(\sum_{d\in [D]} \kl(\lawa_d,\lawb_d)\right)^2}
\end{align*} which concludes the proof.

\subsection{Link between $m$, $\sigma_d$, and $\kl$}\label{appsect:mandsigma}

We have $\sigma_d = \left| \log\frac{\lawa_d(1-\lawb_d)}{\lawb_d(1-\lawa_d)}\right|$ and $m(\lawa,\lawb)=\sum_{d\in [D]} m_d(\lawa,\lawb)$ with $m_d(\lawa,\lawb)=\max\left\{  \log\frac{1-\lawa_d}{1-\lawb_d}, \log\frac{\lawa_d}{\lawb_d}\right\}$.

Fix some dimension $d$. If $\lawb_d\geq \lawa_d$, then $$0\leq m_d(\lawa,\lawb) = \log\frac{1-\lawa_d}{1-\lawb_d}\leq \log\frac{1-\lawa_d}{1-\lawb_d} + \log\frac{\lawb_d}{\lawa_d} = \sigma_d .$$
Conversely, if $\lawb_d \leq \lawa_d$, then $$0\leq m_d(\lawa,\lawb) = \log\frac{\lawa_d}{\lawb_d}\leq \log\frac{\lawa_d}{\lawb_d} + \log\frac{1-\lawb_d}{1-\lawa_d} = \sigma_d .$$

Combining these two results yields \Cref{eq:mandsigma}.

Moreover, $x\in (0,1)\mapsto x^{\lawa_d+1}$, $x\in (0,1)\mapsto \frac{1}{1-x}\in \R_+$ and $x\in \R_+\mapsto x^{\lawa_d}$ are increasing functions; therefore, $x\in (0,1)\mapsto \frac{x^{\lawa_d+1}}{(1-x)^\lawa_d}$ is increasing, and for all $x>\lawa_d$, \[
\frac{x^{\lawa_d+1}}{(1-x)^{\lawa_d}} \geq \frac{\lawa_d^{\lawa_d+1}}{(1-\lawa_d)^{\lawa_d}}
\]
so that, with $x=\lawb_d$, if $\lawb_d > \lawa_d$,
\begin{align*}
	\frac{\lawb_d^{\lawa_d+1}}{(1-\lawb_d)^{\lawa_d}} &\geq \frac{\lawa_d^{\lawa_d+1}}{(1-\lawa_d)^{\lawa_d}} \\
	\frac{\lawb_d(1-\lawa_d)}{\lawa_d(1-\lawb_d)} &\geq \left( \frac{\lawa_d}{\lawb_d}\right)^{\lawa_d}\left( \frac{1-\lawa_d}{1-\lawb_d}\right)^{1-\lawa_d}\\
	\log 	\frac{\lawb_d(1-\lawa_d)}{\lawa_d(1-\lawb_d)} &\geq \lawa_d\log \frac{\lawa_d}{\lawb_d} + (1-\lawa_d)\log\frac{1-\lawa_d}{1-\lawb_d}\\
	\sigma_d&\geq \kl(\lawa_d,\lawb_d)
\end{align*}
Conversely, if $\lawb_d\leq \lawa_d$, we have $\sigma_d \geq \kl(\lawa_d,\lawb_d)$, so that in the end we get \Cref{eq:sigmaandkl}.

\section{Proof of concentration for POE}\label{appsect:proofsPOE}

\subsection{An auxiliary lemma}
\begin{lemma}\label{lem:auxPOE}
	Suppose that $X^{:T}\sim P_\lawb^{\otimes T}$ with parameter $\lawb\in\Param'\subseteq \Param$. Let 
	$$g_{t,d}(\param) = X_d^t \log(\param_d) + (1-X_d^t)\log(1-\param_d) .$$ Then, we have
	\begin{align*}\log\left(\int\prod_{d=1}^D \param_d^{T\overline{X}_d}(1-\param_d)^{T(1-\overline{X}_d)} \d\pi(\param)\right)&\\
		&\hspace{-8em} \ge  \sum_{t=1}^T \sum_{d=1}^Dg_{t,d}(\lawb)-1 - \log\left(\frac{1}{\pi\left(\mathcal{B}_{\infty}\left( \left( 1-\frac 1{TD}\right)\lawb, \frac{1}{TD}\right)\cap \Param'\right)}\right)\end{align*} where $\overline{X}_d = \frac 1 T \sum_{t=1}^T X^t_d$.
	
	In particular, if $\pi$ is the uniform distribution on $\Param' \subset [0,1]^D$ and $\mathcal{B}_{\infty}\left( \left( 1-\frac 1{TD}\right)\lawb, \frac{1}{TD}\right) \subset \Param'$, we get
	$$\log\left(\int\prod_{d=1}^D \param_d^{T\overline{X}_d}(1-\param_d)^{T(1-\overline{X}_d)} \pi(\param)\d \param\right) \ge  \sum_{t=1}^T \sum_{d=1}^Dg_{t,d}(\lawb)-1 - D\log\left(TD\right).$$
\end{lemma}
\begin{proof}
	The proof mimic the one of Lemma F.1 from \cite{agrawalOptimalBestarmIdentification2021}. 
	For any distribution $r$ over $\Param$, we have by the Donsker-Varadhan representation of the KL divergence, 
	\begin{align*}
		\KL(r,\pi) &\ge \E_{\param \sim r}\left[\sum_{t=1}^T \sum_{d=1}^D g_{t,d}(\param) \right] - \log\left(\E_{\param \sim \pi}\left[\exp\left(\sum_{t=1}^T \sum_{d=1}^D g_{t,d}(\param)\right)\right]\right).
	\end{align*}
	Hence, 
	\begin{align*}
		\log\left(\int\prod_{d=1}^D \param_d^{T\overline{X}_d}(1-\param_d)^{T(1-\overline{X}_d)} \d \pi(\param)\right)&=   \log\left(\E_{\param \sim \pi}\left[\exp\left(\sum_{t=1}^T \sum_{d=1}^D g_{t,d}(\param)\right)\right]\right)\\
		&\ge \E_{\param \sim r}\left[\sum_{t=1}^T \sum_{d=1}^D g_{t,d}(\param) \right] - \KL(r,\pi).
	\end{align*}
	Fix $\alpha\in (0,1)$. Now set $\widetilde{\Param}=\{ (1-\alpha)\lawb + \alpha \param, \ \param \in \Param\}$. Because $\Param$ is convex, we have that $\widetilde{\Param}\subset \Param$ and we consider $r$ the distribution $\pi$ restricted to $\widetilde{\Param}$. We have that 
	$$\KL\left(r, \pi\right)=\frac{1}{\pi(\widetilde{\Param})}\int_{\widetilde{\Param}} \log\left(\frac{1}{\pi(\widetilde{\Param})}\right)\pi(\param)\d \param = \log\left(\frac{1}{\pi(\widetilde{\Param})}\right)  $$
	And we have, for any $(1-\alpha)\lawb + \alpha \param \in \widetilde{\Param}$,
	$$\exp\left(g_{t,d}((1-\alpha)\lawb + \alpha \param)\right)  = ((1-\alpha)\lawb_d + \alpha \param_d)^{X_{d}^t} (1-((1-\alpha)\lawb_d + \alpha \param_d))^{1-X_d^t} .$$
	
	If $X_d^t = 0$, this is equal to $$1-((1-\alpha)\lawb_d + \alpha \param_d) \ge 1-((1-\alpha)\lawb_d + \alpha) \ge (1-\alpha)(1 - \lawb_d) ;$$ and, if $X_d^t = 1$, this is equal to $$(1-\alpha)\lawb_d + \alpha \param_d \geq (1-\alpha)\lawb_d. $$ In both cases we have
	$$\exp\left(g_{t,d}((1-\alpha)\lawb + \alpha \param)\right) \ge (1-\alpha) \exp(g_{t,d}(\lawb)).$$
	Hence,
	\begin{align*}
		&\log\left(\int\prod_{d=1}^D \param_d^{T\overline{X}_d}(1-\param_d)^{T(1-\overline{X}_d)} \d \pi(\param) \right) \\
		&\hspace{6em} \ge  \sum_{t=1}^T \sum_{d=1}^Dg_{t,d}(\lawb)-TD \log\left(\frac{1}{1-\alpha}\right) - \log\left(\frac{1}{\pi(\widetilde{\Param})}\right).
	\end{align*}
	Now take $\alpha = 1/(TD+1)$, we get $TD \log\left(\frac{1}{1-\alpha}\right) = TD \log\left(1+ \frac{1}{TD}\right) \le 1$ and, with $\widetilde{\Param} \subseteq \mathcal{B}_{\infty}\left( \left( 1-\frac 1{TD}\right)\lawb, \frac{1}{TD}\right)\cap \Param'$, we get the final result.
	
\end{proof}

\subsection{Proof of Theorem \ref{th:POE}}

We want to bound 
$$g(X^{:T})= \log \left( \prod_{t=1}^T P_{W_{0,t}}(X^t)\right) = \sum_{t=1}^T \log\left(  \int P_\param (X^t) \d W_{0,t}(\param)\right).$$
We have that $g$ is a function of $TD$ independent random variables, the goal is then to apply McDiarmid's inequality. To do that, fix some $t'$ and $d'$. Consider $\tilde{X}^1,\dots,\tilde{X}^T$ where $\tilde{X}^t_d = X^t_d$ for $t,d \neq t',d'$ and $\tilde{X}^{t'}_{d'}$ arbitrary. Remark that, since $\varepsilon\le \param_d \le 1-\varepsilon$ for all $d$, we have
\begin{align*} \varepsilon\int \prod_{d\neq d'}^D \param_d^{X^{t'}_d} (1-\param_d)^{1-X^{t'}_d} \d W_{0,t}(\param)  &\le \int \prod_{d=1}^D \param_d^{X^{t'}_d} (1-\param_d)^{1-X^{t'}_d} \d W_{0,t}(\param)  \\
	& \le \int \prod_{d\neq d'}^D \param_d^{X^{t'}_d} (1-\param_d)^{1-X^{t'}_d} \d W_{0,t}(\param) \end{align*}
and similarly, we can bound, replacing $X$ with $\tilde{X}$,
\begin{align*} \varepsilon\int \prod_{d\neq d'}^D \param_d^{X^{t'}_d} (1-\param_d)^{1-X^{t'}_d} \d W_{0,t}(\param)  &\le \int \prod_{d=1}^D \param_d^{\tilde{X}^{t'}_d} (1-\param_d)^{1-\tilde{X}^{t'}_d} \d W_{0,t}(\param)  \\
	&\le \int \prod_{d\neq d'}^D \param_d^{X^{t'}_d} (1-\param_d)^{1-X^{t'}_d} \d W_{0,t}(\param) \end{align*}
Hence, 
$$\varepsilon \le \frac{\int \prod_{d=1}^D \param_d^{\tilde{X}^{t'}_d} (1-\param_d)^{1-\tilde{X}^{t'}_d} \d W_{0,t}(\param) }{\int \prod_{d=1}^D \param_d^{X^{t'}_d} (1-\param_d)^{1-X^{t'}_d} \d W_{0,t}(\param)} \le \frac{1}{\varepsilon} .$$
So that we get
$$|g(X^{:T}) - g(\tilde{X}^{:T})| \le \log(1/\varepsilon).$$
By McDiarmid inequality, this means that with probability larger than $1-\delta$, we have
\begin{align*} g(X^{:T}) &\ge \E_Q\left[g(X^{:T})\right] -\sqrt{\frac{ \sum_{n,d} \log(1/\varepsilon)^2 \log(1/\delta)}{2}} \\
	&=  \E_Q\left[g(X^{:T})\right] -\log(1/\varepsilon)\sqrt{\frac{TD \log(1/\delta)}{2}}. \end{align*}
Then, we control the expectation. 
\begin{align*}
	\E_{X^{:T}\sim Q^{\otimes T}}\left[g(X^{:T})\right]&= \E_{Q^{\otimes T}}\left[\log \left( \prod_{t=1}^T P_{W_{0,t}}(X^t)\right)\right]\\
	&= \sum_{t=1}^T\E_Q\left[\log \left(P_{W_{0,t}}(X^t)\right)\right]
\end{align*}
Hence, with probability larger than $1-\delta$
$$\log \left( \prod_{t=1}^T P_{W_{0,t}}(X^t)\right) \ge \sum_{t=1}^T\E_Q\left[\log \left(P_{W_{0,t}}(X^t)\right)\right] - \log(1/\varepsilon)\sqrt{\frac{ TD \log(1/\delta)}{2}} $$
Similarly, we also have with probability larger than $1-\delta$,
$$\log \left( \prod_{t=1}^T P_{W_{1,t}}(X^t)\right) \le  \sum_{t=1}^T\E_Q\left[\log \left(P_{W_{1,t}}(X^t)\right)\right] + \log(1/\varepsilon)\sqrt{\frac{TD \log(1/\delta)}{2}} $$
Hence, with probability larger than $1-2\delta$, 
\begin{align*}
	\log \left( \prod_{t=1}^T \frac{d\hat{Q}_t(X^t)}{d\hat{P}_t(X^t)}\right) &\ge \sum_{t=1}^T\E_Q\left[\log \left(P_{W_{1,t}}(X^t)\right)\right]-\sum_{t=1}^T\E_Q\left[\log \left(P_{W_{0,t}}(X^t)\right)\right] \\
	&\hspace{6em}-\log(1/\varepsilon)\sqrt{2 TD \log(1/\delta)}\\
	&= \sum_{t=1}^T \left(- \KL\left(Q, P_{W_{1,t}}\right) + \KL\left(Q, P_{W_{0,t}}\right)\right)-\log\frac 1 \varepsilon \sqrt{2 TD \log\frac 1 \delta}
\end{align*}
Then, remark that $P^\star$ satisfies $P^\star \in \arg\inf_{P} \KL(Q,P)$ by definition of the RIPr; hence, $\KL(Q,P_{W_{0,n}}) \ge \KL(Q,P^\star)$ and 
\begin{align*}
	\log \left( \prod_{t=1}^T \frac{\d\hat{Q}_t(X^t)}{\d\hat{P}_t(X^t)}\right) \ge& \sum_{t=1}^T \left(- \KL\left(Q, P_{W_{1,t}}\right) + \KL\left(Q, P^*\right)\right)-\log(1/\varepsilon)\sqrt{2 TD \log(1/\delta)}
\end{align*}
We then want to upper bound
\begin{align*}
	\sum_{t=1}^T  \KL\left(Q, P_{W_{1,t}}\right)&= \E_Q\left[ \log\left(\prod_{t=1}^T\frac{Q(X^t)}{P_{W_{1,t}}(X^t)}\right)\right]\\
	&= \E_Q\left[ \log\left(\frac{\prod_{d=1}^D \mu_d^{T\overline{X}_d}(1-\mu_d)^{T(1-\overline{X}_d)}}{\int\prod_{d=1}^D \param_d^{T\overline{X}_d}(1-\param_d)^{T(1-\overline{X}_d)} \d \pi(\param)}\right)\right]\\
\end{align*}
and we conclude with \Cref{lem:auxPOE}.

\section{Proof for the fixed sample size tests}

\subsection{Proof of \Cref{eq:exptau}}\label{appsect:GROexp}

We finish here the proof for \Cref{th:GRO}.

We start by assuming that the data was collected according to $Q\in\Hyp_1$, and denote $\lawb$ the parameter of the distribution of its RIPr onto $\Hyp_0$.

From Lemma~\ref{lem:GRO},
with $\gamma\in (0,1)$,
\[ \P_\Lawa \left( \tau \ge \frac{\log\frac{1}{\alpha} +m( \lawa,\lawb)}{(1-\gamma)\sum_{d\in [D]}  \kl(\lawa_d,\lawb_d)}+\frac{\log (1/\nu) \sum_{d\in [D]} \sigma_d^2}{8\gamma(1-\gamma)\left(\sum_{d\in [D]} \kl(\lawa_d,\lawb_d)\right)^2}
\right) \leq \nu \]
so that, with $\nu_t = \exp\left( \left(\frac{\log(1/\alpha) +m( \lawa,\lawb)}{(1-\gamma)\sum_d \kl(\lawa_d,\lawb_d)} -t\right) \frac{8\gamma(1-\gamma)\left(\sum_d \kl(\lawa_d,\lawb_d)\right)^2}{\sum_d \sigma_d^2}\right)$, \[ \P_\Lawa\left( \tau\geq t\right)\leq \nu_t\]

So that{\small \begin{align*} 
		\E_\Lawa[\tau] &\leq T + \sum_{t=T}^\infty \nu_t \\
		&\leq T+\sum_{t=T}^\infty \exp\left( \left( \frac{\log(1/\alpha) +m( \lawa,\lawb)}{(1-\gamma)\sum_d \kl(\lawa_d,\lawb_d)} -t\right) \frac{8\gamma(1-\gamma)\left(\sum_d \kl(\lawa_d,\lawb_d)\right)^2}{\sum_d \sigma_d^2}\right) \\
		&\leq T+ \exp\left(  \frac{8\gamma\left(\log(1/\alpha) +m( \lawa,\lawb)\right) \sum_d \kl(\lawa_d,\lawb_d)}{\sum_d \sigma_d^2}\right) \sum_{t=T}^\infty \exp\left( -8t\gamma(1-\gamma)\frac{\left(\sum_d \kl(\lawa_d,\lawb_d)\right)^2}{\sum_d \sigma_d^2}\right) \\
		&\leq T + \exp\left(  \frac{8\gamma\left(\log(1/\alpha) +m( \lawa,\lawb)\right) \sum_d \kl(\lawa_d,\lawb_d)}{\sum_d \sigma_d^2}\right) \frac{\exp\left( -8T\gamma(1-\gamma)\frac{\left(\sum_d \kl(\lawa_d,\lawb_d)\right)^2}{\sum_d \sigma_d^2}\right)}{1-\exp\left( -8\gamma(1-\gamma)\frac{\left(\sum_d \kl(\lawa_d,\lawb_d)\right)^2}{\sum_d \sigma_d^2}\right)}
\end{align*}}

Taking $T=\frac{1}{1-\gamma} \frac{\log 1/\alpha}{\sum_d \kl(\lawa_d,\lawb_d)}$, \small{
	\begin{align*}
		\E_\Lawa[\tau]&\leq \frac{1}{1-\gamma} \frac{\log 1/\alpha}{\sum_d \kl(\lawa_d,\lawb_d)} \\
		&\hspace{2em}  + \exp\left(  \frac{8\gamma\left(\log(1/\alpha) +m( \lawa,\lawb)\right) \sum_d \kl(\lawa_d,\lawb_d)}{\sum_d \sigma_d^2}\right) \frac{\exp\left( -8 \log (1/\alpha)\gamma\frac{\sum_d \kl(\lawa_d,\lawb_d)}{\sum_d \sigma_d^2}\right)}{1-\exp\left( -8\gamma(1-\gamma)\frac{\left(\sum_d \kl(\lawa_d,\lawb_d)\right)^2}{\sum_d \sigma_d^2}\right)}\\
		&\leq \frac{1}{1-\gamma} \frac{\log 1/\alpha}{\sum_d \kl(\lawa_d,\lawb_d)} + \frac{\exp\left(  8\gamma\frac{m( \lawa,\lawb) \sum_d \kl(\lawa_d,\lawb_d)}{\sum_d \sigma^2_d} \right)}{1-\exp\left( -8\gamma(1-\gamma)\frac{\left(\sum_d \kl(\lawa_d,\lawb_d)\right)^2}{\sum_d \sigma_d^2}\right)}
\end{align*}}

Taking $\gamma = \sqrt{\frac{\sum_d\kl(\lawa_d,\lawb_d)}{\log(1/\alpha)}}$ for $\alpha$ small enough that $\gamma\leq \frac 1 2$, the right hand term becomes
\begin{align*}
	&\left(1+\sqrt{\frac{\sum_d\kl(\lawa_d,\lawb_d)}{\log(1/\alpha)}} + \cO  \left( \frac{\sum_d\kl(\lawa_d,\lawb_d)}{\log(1/\alpha)}\right) \right) \frac{\log(1/\alpha)}{\sum_d \kl(\lawa_d,\lawb_d)} 
	\\ &\hspace{4em} + 
	\frac{\exp\left( 8 \frac{  m( \lawa,\lawb)\left( \sum_d \kl(\lawa_d,\lawb_d)\right)^{3/2}}{\sqrt{\log(1/\alpha)}\sum_d \sigma_d^2}\right)}
	{1-\exp\left( 4\frac{\left(\sum_d \kl(\lawa_d,\lawb_d)\right)^{5/2}}{\sqrt{\log(1/\alpha)}\sum_d \sigma_d^2}
		\right)}\\
	= & \frac{\log(1/\alpha)}{\sum_d \kl(\lawa_d,\lawb_d)} + \sqrt{\frac{\log(1/\alpha)}{\sum_d \kl(\lawa_d,\lawb_d)}} + \cO(1) \\
	&\hspace{4em} +	\frac{1+ 8 \frac{  m( \lawa,\lawb)\left( \sum_d \kl(\lawa_d,\lawb_d)\right)^{3/2}}{\sqrt{\log(1/\alpha)}\sum_d \sigma_d^2} + \cO\left( \frac{  m( \lawa,\lawb)^2\left( \sum_d \kl(\lawa_d,\lawb_d)\right)^{3}}{\log(1/\alpha)\left(\sum_d \sigma_d^2\right)^2} \right)}
	{ 4\frac{\left(\sum_d \kl(\lawa_d,\lawb_d)\right)^{5/2}}{\sqrt{\log(1/\alpha)}\sum_d \sigma_d^2}
		+ \cO\left( \frac{\left(\sum_d \kl(\lawa_d,\lawb_d)\right)^{5}}{\log(1/\alpha)\left(\sum_d \sigma_d^2\right)^2}
		\right)
	} 
\end{align*} which then yields 
\begin{align*}
	& \frac{\log(1/\alpha)}{\sum_d \kl(\lawa_d,\lawb_d)} + \sqrt{\frac{\log(1/\alpha)}{\sum_d \kl(\lawa_d,\lawb_d)}} + \cO(1) + \frac{\sqrt{\log(1/\alpha)}\sum_d \sigma_d^2}{4\left(\sum_d \kl(\lawa_d,\lawb_d)\right)^{5/2}} +\cO\left(1 \right)\\
	&\hspace{4em} + \frac{  2m( \lawa,\lawb) + \cO\left( \frac{  m( \lawa,\lawb)^2\left( \sum_d \kl(\lawa_d,\lawb_d)\right)^{3/2}}{\sqrt{\log(1/\alpha)}\sum_d \sigma_d^2} \right)}
	{\sum_d \kl(\lawa_d,\lawb_d)  + \cO\left( \frac{\left(\sum_d \kl(\lawa_d,\lawb_d)\right)^{7/2}}{\sqrt{\log(1/\alpha)}\sum_d \sigma_d^2} \right)} \\
\end{align*} so that \[\E_\Lawa[\tau] = \frac{\log(1/\alpha)}{\sum_d \kl(\lawa_d,\lawb_d)} + \cO\left(\sqrt{\frac{\log(1/\alpha)}{\sum_d \kl(\lawa_d,\lawb_d)}}\right) \]

If the data is collected under $\Hyp_0$ instead, the same computations can be run. Since the stopping time of the combined test is the minimum of the two stopping times, we have the result of \Cref{eq:exptau}.

\subsection{Proof of Theorem \ref{th:POE_samplesize}}\label{appsect:fixedPOE}

By \Cref{th:POE}, to achieve power $1-\beta$, it suffices that \[ T \KL\left(Q, P^*\right)-1-\log\left(\frac{1}{\prior\left(\mathcal{B}_{\infty}\left( \left( 1-\frac 1{TD}\right)\lawa, \frac{1}{TD}\right)\cap \Param_1 \right)}\right)  -\log(1/\varepsilon)\sqrt{2 TD \log(2/\beta)} \geq \log 1/\alpha .\]
Let us assume the converse. We then have \begin{align*}
	T \KL\left(Q, P^*\right)-1-\log\left(\frac{1}{\prior\left(\mathcal{B}_{\infty}\left( \left( 1-\frac 1{TD}\right)\lawa, \frac{1}{TD}\right)\cap \Param_1 \right)}\right)  -\log(1/\varepsilon)\sqrt{2 TD \log(2/\beta)} &\leq \log 1/\alpha \\
	T \KL\left(Q, P^*\right)-\left( 1+\log\left(\frac{1}{\prior\left(\mathcal{B}_{\infty}\left( \left( 1-\frac 1{TD}\right)\lawa, \frac{1}{TD}\right)\cap \Param_1 \right)}\right) +\log 1/\alpha\right) &\leq \log(1/\varepsilon)\sqrt{2 TD \log(2/\beta)}
\end{align*} so that, fixing $B=1+\log\left(\frac{1}{\prior\left(\mathcal{B}_{\infty}\left( \left( 1-\frac 1{TD}\right)\lawa, \frac{1}{TD}\right)\cap \Param_1 \right)}\right) +\log 1/\alpha$ and squaring, since $T$ is chosen so that the left-hand side is positive,
\[T^2 \KL(Q,P^\star)^2 -T\left(2\KL(Q,P^\star)B+2D\ln\frac{2}{\beta}\left(\log\frac{1}{\varepsilon}\right)^2\right)
+B^2 \leq 0\] and
{\smaller 
	\[
T\leq \frac{2\KL(Q,P^\star)B +2D\ln\frac{2}{\beta}\left(\log\frac{1}{\varepsilon}\right)^2 +\sqrt{ \left( 2\KL(Q,P^\star)B +2D\ln\frac{2}{\beta}\log\left(\frac{1}{\varepsilon}\right)^2\right)^2 -4\KL(Q,P^\star)^2B^2}}{2\KL(Q,P^\star)^2}, \] }
which yields {\smaller
\begin{align*}
\KL(Q,P^\star)^2 T&\leq \KL(Q,P^\star)B +D\ln\frac{2}{\beta}\left(\log\frac{1}{\varepsilon}\right)^2 \\
& \hspace{1.5em} +\sqrt{ 2\KL(Q,P^\star)B D\ln\frac{2}{\beta}\log\left(\frac{1}{\varepsilon}\right)^2 + D^2\left(\ln\frac{2}{\beta}\right)^2 \left(\log\frac{1}{\varepsilon}\right)^4 } \\
 &\leq \KL(Q,P^\star)B +D\ln\frac{2}{\beta}\left(\log\frac{1}{\varepsilon}\right)^2 \\
& \hspace{1.5em} +\log\frac{1}{\varepsilon}\sqrt{D\ln\frac{2}{\beta}\left( 2\KL(Q,P^\star)B  + D\ln\frac{2}{\beta}\left(\log\frac{1}{\varepsilon}\right)^2\right) }\end{align*} } and finally  \begin{align*}
T&\leq \frac{B' +\log\frac{1}{\alpha}}{\KL(Q,P^\star)} + \frac{D\ln\frac{2}{\beta}\left(\log\frac{1}{\varepsilon}\right)^2}{\KL(Q,P^\star)^2} \\ 
&\hspace{2em}+ \frac{\log\frac{1}{\varepsilon}\sqrt{D\ln\frac{2}{\beta}}}{\KL(Q,P^\star)^2}\sqrt{ 2\KL(Q,P^\star)\left(B' +\log\frac{1}{\alpha}\right)  + D\ln\frac{2}{\beta}\left(\log\frac{1}{\varepsilon}\right)^2 }
\end{align*} with $B' = 1+\log\left(\frac{1}{\prior\left(\mathcal{B}_{\infty}\left( \left( 1-\frac 1{TD}\right)\lawa, \frac{1}{TD}\right)\cap \Param_1 \right)}\right)$,
which yields the theorem.

\section{Proofs for fixed confidence tests: proof of Theorem \ref{th:auxfc}}\label{appsect:proofauxfc}

We focus on the case $i=1$; the case $i=0$ works the exact same way.

From \Cref{th:POE}, we have that, if $T$ satisfies $$T\KL(Q,P^\star)-1-\log\frac{1}{\prior\left(\mathcal{B}_{\infty}\left( \left( 1-\frac 1{TD}\right)\lawa, \frac{1}{TD}\right)\cap \Param_i \right)} -\log\frac{1}{\alpha}\geq 0$$ then, with $\beta_T = 2\exp\left[
-\left(\frac{T\KL(Q,P^\star)-1-\log\frac{1}{\prior\left(\mathcal{B}_{\infty}\left( \left( 1-\frac 1{TD}\right)\lawa, \frac{1}{TD}\right)\cap \Param_1 \right)} -\log\frac{1}{\alpha}}{\log\frac{1}{\varepsilon}\sqrt{2TD}}\right)^2 \right]$, then, under distribution $Q$, with probability larger than $1-\beta_T$, {\smaller
\begin{align*}
	\log \left( \prod_{t=1}^T \frac{d\hat{Q}_t(X^t)}{d\hat{P}_t(X^t)}\right) &\ge T \KL\left(Q, P^*\right)-1-\log\left(\frac{1}{\prior\left(\mathcal{B}_{\infty}\left( \left( 1-\frac 1{TD}\right)\lawa, \frac{1}{TD}\right)\cap \Param_1 \right)}\right)  -\log(1/\varepsilon)\sqrt{2 TD \log(2/\beta)}\\
	&\ge T \KL\left(Q, P^*\right)-1-\log\left(\frac{1}{\prior\left(\mathcal{B}_{\infty}\left( \left( 1-\frac 1{TD}\right)\lawa, \frac{1}{TD}\right)\cap \Param_1 \right))}\right)  \\ 
	&\hspace{2em}-\log(1/\varepsilon)\sqrt{2 TD} \left(\frac{T\KL(Q,P^\star)-1-\log\left(\frac{1}{\prior\left(\mathcal{B}_{\infty}\left( \left( 1-\frac 1{TD}\right)\lawa, \frac{1}{TD}\right)\cap \Param_1 \right))}\right) -\log\frac{1}{\alpha}}{\log\frac{1}{\varepsilon}\sqrt{2TD}}\right) \\
	&\ge \log\frac{1}{\alpha}
\end{align*} }meaning that the test has power larger than $1-\beta_T$ at step $T$: $\P_Q(\tau\geq T) \leq \beta_T$.

We therefore have, for any $\gamma \in (0,1)$, for any $T$ that satisfies $$T\KL(Q,P^\star) -1-\log\left(\frac{1}{\prior\left(\mathcal{B}_{\infty}\left( \left( 1-\frac 1{TD}\right)\lawa, \frac{1}{TD}\right)\cap \Param_1 \right)}\right)-\log\frac{1}{\alpha} \geq \gamma T \KL(Q,P^\star) , $$ we have
\begin{align*}
	\E_Q[\tau] &\leq T+ \sum_{t=T+1}^{+\infty} \beta_t \\
	&\leq T+ 2\sum_{t=T+1}^{+\infty} \exp\left[
	-\left(\frac{\gamma t\KL(Q,P^\star)}{\log\frac{1}{\varepsilon}\sqrt{2tD}}\right)^2 \right]\\
	&\leq T+ \frac{2\exp\left[ -\frac{\gamma^2 \KL(Q,P^\star)}{\left(\log \frac{1}{\varepsilon}\right)^2 2 D}T\right]}{1-\exp\left[ -\frac{\gamma^2 \KL(Q,P^\star)}{\left(\log \frac{1}{\varepsilon}\right)^2 2 D}\right]}\\
	&\leq T+\frac{2}{1-\exp\left[ -\frac{\gamma^2 \KL(Q,P^\star)}{\left(\log \frac{1}{\varepsilon}\right)^2 2 D}\right]}
\end{align*}

\section{Proof of Lemma~\ref{lem:deriv_fw}}\label{sec:proof_deriv_fw}
	\begin{align*}
		\KL(P_{W_1},(1-t)P_{W_0} +tP_{W'}) &= \E_{X\sim P_{W_1}}\left[  \log \frac{\d P_{W_1}(X)}{(1-t)P_{W_0}(X) +tP_{W'}(X)}\right] \\
		\frac{\d}{\d t} \KL(P_{W_1},(1-t)P_{W_0} +tP_{W'}) &= - \E_{P_{W_1}} \left[ \frac{ -P_{W_0} +P_{W'}}{(1-t)P_{W_0} +tP_{W'}} \right]
	\end{align*} so that, for $t=0$,
	\begin{align*} \frac{\d}{\d t} \KL(P_{W_1},(1-t)P_{W_0} + t P_{W'})\mid_{t=0} &= -\E_{P_{W_1}} \left[ \frac{ -P_{W_0} +P_{W'}}{P_{W_0}} \right] \\
	&= 1- \E_{P_{W_1}} \left[ \frac{ P_{W'}}{P_{W_0}} \right]  \end{align*} and we can conclude by noting that $\E_{P_{W_1}} \left[ \frac{ P_{W'}}{P_{W_0}} \right] = \E_{P_{W'}} \left[ \frac{ P_{W_1}}{P_{W_0}} \right]$.
	
	Finally,
	\begin{align*} \argmin_{P_{W'}} \nabla \KL (P_{W_1},P_{W_0}) ^\top P_{W'} &= \argmin_{P_{W'}} \nabla \KL (P_{W_1},P_{W_0}) ^\top (P_{W'}-P_{W_0}) \\ 
		&= \argmin_{P_{W'}} \frac{\d}{\d t} \KL(P_{W_1},(1-t)P_{W_0} + t P_{W'})\mid_{t=0}\\  
		&= \argmin_{P_{W'}} -\E_{P_{W'}} \left[ \frac{P_{W_1}}{P_{W_0}}\right]  \end{align*} and we can conclude.

\end{document}